\documentclass[11pt]{article}
\usepackage{amsmath,amssymb,amsthm,enumerate,color,graphicx,natbib,a4wide}
\usepackage[active]{srcltx}

\usepackage[colorlinks]{hyperref}

\def\esp{\mathbb E}
\def\pr{\mathbb P}
\def\var{\mathrm{var}}
\def\cov{\mathrm{cov}}

\def\covmdep{\mathcal L}
\def\limitlrd{\aleph}

\newtheorem{theorem}{Theorem}
\newtheorem{lemma}[theorem]{Lemma}
\newtheorem{claim}{Claim}
\newtheorem{corollary}[theorem]{Corollary}

\newtheorem{hypothesis}{Assumption}
\theoremstyle{remark}
\newtheorem{remark}{Remark}
\newtheorem{example}{Example}

\begin{document}

\title{Estimation of limiting conditional distributions for the heavy tailed
  long memory stochastic volatility process}

\author{Rafa{\l} Kulik\thanks{University of Ottawa} \and Philippe
  Soulier\thanks{Universit\'e de Paris-Ouest}}

\date{}
\maketitle

\begin{abstract}
  We consider Stochastic Volatility processes with heavy tails and possible long
  memory in volatility. We study the limiting conditional distribution of future
  events given that some present or past event was extreme (i.e. above a level
  which tends to infinity). Even though extremes of stochastic volatility
  processes are asymptotically independent (in the sense of extreme value
  theory), these limiting conditional distributions differ from the
  i.i.d. case. We introduce estimators of these limiting conditional
  distributions and study their asymptotic properties. If volatility has long
  memory, then the rate of convergence and the limiting distribution of the
  centered estimators can depend on the long memory parameter (Hurst index).
\end{abstract}

\section{Introduction}
One of the empirical features of financial data is that log-returns are
uncorrelated, but their squares, or absolute values, are dependent, possibly
with long memory. Another important feature is that log-returns are
heavy-tailed. There are two common classes of processes to model such behaviour:
the generalized autoregressive conditional heteroscedastic (GARCH) process and
the stochastic volatility (SV) process; the latter introduced by
\cite{breidt:crato:delima:1998} and \cite{harvey:1998}. The former
class of models rules out long memory in the squares, while the latter allows
for it. We will therefore concentrate in this paper on the class of SV
processes, which we define now.

Let $\{Y_j,j\in\mathbb Z\}$ be the observed process (e.g. log-returns of some
financial time series), and assume that it can be expressed as
\begin{gather}
  \label{eq:def-volstoch}
  Y_j = \sigma(X_j) Z_j  \; .
\end{gather}
where $\sigma$ is some (possibly unknown) positive function, $\{Z_j,j \in\mathbb
Z\}$ is an i.i.d. sequence and $\{X_j, j \in \mathbb Z\}$ is a stationary
Gaussian process with mean zero, unit variance, autocovariance function
$\{\gamma_n\}$, and independent from the i.i.d. sequence.  The sequence
$\sigma(X_j)$ can be seen as a proxy for the volatility. We will assume that
either $\{X_j\}$ is weakly dependent in the sense that
\begin{align}
  \label{eq:weakdep}
  \sum_{j=1}^\infty |\gamma_j| < \infty \; ,
\end{align}
or that it has  long memory with Hurst index $H \in (1/2,1)$, i.e.
\begin{gather}
  \label{eq:lrd}
  \gamma_n = \cov(X_0,X_n) = n^{2H-2} \ell(n)
\end{gather}
where $\ell$ is a slowly varying function.

Furthermore, we assume that the marginal distribution $F_Z$ of the i.i.d.
sequence $\{Z_j\}$ has a regularly varying right tail with index $\alpha>0$,
i.e., for all positive $y$,
\begin{gather}
  \label{eq:attraction}
  \lim_{t\to\infty} \pr(Z> ty \mid Z > t) = \lim_{t\to\infty} \frac{\bar
    F_Z(ty)}{\bar F_Z(t)} = y^{-\alpha} \; .
\end{gather}
Examples of heavy tailed distributions include the stable distributions with
index $\alpha\in(0,2)$, the $t$ distribution with $\alpha$ degrees of freedom,
and the Pareto distribution with index $\alpha$.

By Breiman's lemma \cite{breiman:1965,resnick:2007}, if
$\esp[\sigma^{\alpha+\epsilon}(X)]<\infty$ for some $\epsilon>0$, then the
marginal distribution of $\{Y_j\}$ also has a regularly varying right tail with
index $\alpha$ and
\begin{gather}
   \label{eq:breiman}
   \lim_{x\to\infty} \frac{\pr(Y>xy)}{\pr(Z>x)} = \esp[\sigma^\alpha(X)]
   y^{-\alpha} \; ,
 \end{gather}
where $X$, $Y$ and $Z$ denote random variables with the
same joint distribution as $X_0$, $Y_0$ and~$Z_0$.

Estimation and test of the possible long memory of such processes has been
studied by \cite{hurvich:moulines:soulier:2005}. Estimation of the tail of the
marginal distribution by the Hill estimator has been studied in
\cite{kulik:soulier:2011}.

In this paper we are concerned with certain extremal properties of the finite
dimensional joint distributions of the process $\{Y_j\}$ when $Z$ is heavy
tailed and the Gaussian process $\{X_j\}$ possibly has long memory.

From the extreme value point of view, there is a significant distinction between
the GARCH and SV models. In the first one, exceedances over a large threshold
are asymptotically dependent and extremes do cluster.  In the SV model,
exceedances are asymptotically independent.  More precisely, for any positive
integer $m$, and positive real numbers $x,y$,
\begin{gather}
  \label{eq:asymptotic-indep}
  \lim_{t\to\infty} t\pr(Y_0>a(t)x \;, \ Y_m > a(t)y) = 0 \;,
\end{gather}
where $a(t) = F_Z^\leftarrow(1-1/t)$ and $F_Z^\leftarrow$ is the left continuous
inverse of $F_Z$.  This holds since it can be easily shown by a conditioning
argument that
\begin{align}
  \pr(Y_0> t \;, \ Y_m > t) \sim \mathrm{c} \times \pr(Y_0 > t)^2 \; , \ \
  t\to\infty \; , \label{eq:taildependence}
\end{align}
for some positive constant $c$.

The above observations may lead to the incorrect conclusion that, for the SV
process, there is no spillover from past extreme observations onto future values
and from the extremal behaviour point of view we can treat the SV process as an
i.i.d. sequence. However, under the assumptions stated previously, it holds that
\begin{align}
  \label{eq:1}
  \lim_{t\to\infty} \pr(Y_m \le y \mid Y_0>t) =
  \frac{\esp[\sigma^{\alpha}(X_0)F_Z(y/\sigma(X_m))]}
  {\esp[\sigma^{\alpha}(X_m)]} \; .
\end{align}
Therefore, the limiting conditional distribution is influenced by the dependence
structure of the time series. To illustrate this, we show in
Figure~\ref{fig:ecdf} estimates of the standard distribution function and of the
conditional distribution for a simulated SV process. Clearly, the two estimated
distributions are different, as suggested by (\ref{eq:1}). For a comparison, we
also plot the corresponding estimates for i.i.d. data.
\begin{figure}[h]
 \centering
 \includegraphics[height=8cm, width=14cm]{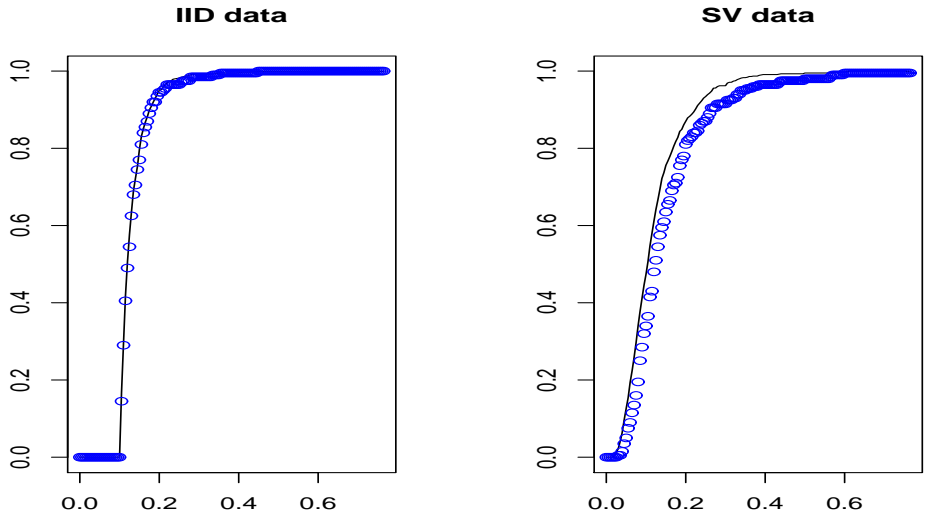}
  \caption{Empirical Conditional Distribution (points) and Empirical
    Distribution (solid line) for SV model (right panel) and i.i.d. data (left
    panel) }
\label {fig:ecdf}
\end{figure}
Other kind of extremal events can be considered, for instance, we may be
interested in the conditional distribution of some future values given that a
linear combination (portfolio) of past values is extremely large, or that two
consecutive values are large.  As in Equation (\ref{eq:1}), in each of these
cases, a proper limiting distribution can be obtained.  To give a general
framework for these conditional distributions, we introduce a modified version
of the extremogram of \cite{davis:mikosch:2009}.  For fixed positive integers
$h<m$ and $h'\geq0$, Borel sets $A \subset \mathbb R^h$ and $B\subset \mathbb
R^{h'+1}$, we are interested in the limit denoted by $\rho(A,B,m)$, if it
exists:
\begin{align}
 \label{eq:def-extremogram}
 \rho(A,B,m) = \lim_{t\to\infty} \pr( (Y_{m},\dots,Y_{m+h'}) \in B \mid
 (Y_{1},\dots,Y_{h}) \in t A) \; .
\end{align}
The set $A$ represents the type of events considered. For instance, if we choose
$A=\{(x,y,z)\in [0,\infty)^3 \mid x+y+z>1\}$, then for large $t$,
$\{(Y_{-2},Y_{-1},Y_{0}) \in t A\}$ is the event that the sum of last three
observations was extremely large. The set $B$ represents the type of future
events of interest.

In the original definition of the extremogram of \cite{davis:mikosch:2009}, the set $B$ is also dilated by $t$. This is well suited to the context of asymptotic dependence, as arises in GARCH processes. But in the context of asymptotic independence, this would yield a degenerate limit: if $h<m$, then for most sets $A$ and $B$,
\begin{align*}
\lim_{t\to\infty} \pr( (Y_{m},\dots,Y_{m+h'}) \in tB \mid
 (Y_{1},\dots,Y_{h}) \in t A) = 0 \; .
\end{align*}

The general aim of this paper is to investigate the existence of these limiting
conditional distributions appearing in~(\ref{eq:def-extremogram}) and their
statistical estimation.  The paper is the first step towards understanding conditional laws for
stochastic volatility models. Although we provide theoretical properties
of estimators, their practical use should be investigated in conjunction
with resampling techniques. This is a topic of authors' current research.\\

The paper is structured as follows. In Section \ref{sec:subcone}, we present a general
framework that enables to treat various examples in a unified way.
%This framework is based on the concepts of \textit{hidden regular variation} and regular variation on subcones introduced in
%\cite{resnick:2002}. See also \cite{resnick:2008}.
In Section \ref{sec:estimation} we present the estimation
procedure with appropriate limiting results.

% An important practical problem is to distinguish between e.g. GARCH processes
% and SV processes, based on empirical data. A squared GARCH process cannot have
% long memory whereas SV process can exhibit long range dependence in
% squares. Therefore, one may think of applying standard second-order type
% techniques based on the autocovariance function (ACF). However, these may not be
% suitable in the case of heavy tailed data, as emphasized in
% \cite{davis:mikosch:2009}.
% %\begin{quote}
% %  The ACF is sometimes overvalued as a tool for measuring dependence, especially
% %  if one is only interested in extremes. It does, of course, determine the
% %  distribution of a stationary Gaussian sequence, but for non-Gaussian and
% %  nonlinear time series, the ACF often provides little insight into the
% %  dependence structure of the process.
% %\end{quote}
% Therefore it might be of interest to distinguish between these two classes of
% models by means of their extremal properties.  To achieve this, in Section
% \ref{Section:Extremal-Independence}, we propose a test of extremal independence.

The proofs are given in Section \ref{sec:proof}. In the Appendix we collect
relevant results on second order regular variation, (long memory) Gaussian
processes, and criteria for tightness.

We conclude this introduction by gathering some notation that will be used
throughout the paper.  We denote convergence in probability by $\to_P$, weak
convergences of sequences of random variables or vectors by $\to_d$ and weak
convergence in the Skorokhod space $\mathcal D(\mathbb{R}^q)$ of cadlag functions
defined on $\mathbb{R}^q$ endowed with the $J_1$ topology by $\Rightarrow$.

Boldface letters denote vectors.  Product of vectors and
inequalities between vectors are taken componentwise: $\mathbf u
\cdot \mathbf v = (u_1v_1,\dots,u_dv_d)$; $\mathbf x \leq \mathbf y$
if and only if $x_i \leq y_i$ for all $i=1,\dots,d$. The (multivariate)
interval $(\boldsymbol\infty,\mathbf y]$ is defined accordingly:
$(\boldsymbol\infty,\mathbf y] = \prod_{i=1}^d (-\infty,y_i]$.

For any univariate process $\{\xi_j\}$ and any integers $h \leq h'$, let
$\boldsymbol{\xi}_{h,h'}$ denote the $(h'-h+1)$-dimensional vector
$(\xi_{h},\dots,\xi_{h'})$.

For $A\subset\mathbb{R}^d$ and $\mathbf u \in (0,\infty)^d$, $\mathbf u^{-1} \cdot A
= \{\mathbf x \in \mathbb{R}^d \mid \mathbf u \cdot \mathbf x \in A\}$.

If $\mathbf X$ is a random vector, we denote by $L^p(\mathbf X)$ the set of
measurable functions $f$ such that $\esp[|f(\mathbf X)|^p]<\infty$.

For any univariate process $\{\xi_j\}$ and any integers $h \leq h'$, let
$\boldsymbol{\xi}_{h,h'}$ denote the $(h'-h+1)$-dimensional vector
$(\xi_{h},\dots,\xi_{h'})$.

The $\sigma$-field generated by the process $\{X_j\}$ is denoted by $\mathcal X$.

\section{Regular variation on subcones}
\label{sec:subcone}

\def\bsigma{\boldsymbol\sigma}
\def\muc{\mu_{\mathcal C}}
\def\nuc{\nu_{\mathcal C}}
\def\gc{g_{\mathcal C}}
\def\Tc{T_{\mathcal C}}
\def\zh{\mathbf Z_{1,h}}
\def\xh{\mathbf X_{1,h}}
\def\cj{\mathcal C_{\mathbf j}}

Since we considered dilated sets $tA$, where $A \subset \mathbb{R}^h$ for some
integer $h>0$, it is natural to consider cones, that is subsets $\mathcal C$ of
$[0,\infty]^h$ such that $tx\in \mathcal C$ for all $x\in \mathcal C$ and
$t>0$. The next definition is related to the concept of regular variation on
cones of \cite{resnick:2008}.  We endow $\mathbb R^h$ with the topology induced
by any norm and $[0,\infty]^h$ is the compactification of $[0,\infty)^h$. A
subset $A$ of $[0,\infty]^h\setminus\{0\}$ is relatively compact if its closure
is compact. See \cite{MR900810} for more details. We first state a general
assumption and will give examples afterwards.

\begin{hypothesis}
  \label{hypo:cone}
  Let $h$ be a fixed positive integer. Let $\mathcal C$ be a subcone of
  $[0,\infty]^h\setminus\{\mathbf 0\}$ such that, (i) for all relatively compact
  subsets $A$ of~$\mathcal C$ and all $\mathbf u \in(0,\infty)^h$, $\mathbf
  u^{-1} \cdot A$ is relatively compact in~$\mathcal C$, and (ii) there exists a
  function $\gc$ and a non degenerate Radon measure $\nuc$ on $\mathcal C$ such that
\begin{align}
  \label{eq:reg-var-C}
  \lim_{t\to\infty} \frac{\pr(\zh \in t A)}{\gc(\bar F_Z(t))} = \nuc(A) \; .
\end{align}

\end{hypothesis}

Note that in the case $h=1$, the cone $\mathcal C = (0,\infty)$ and
Assumption~\ref{hypo:cone} is nothing more than the regular variation of the
tail of $Z_1$.

Assumption~\ref{hypo:cone} implies that the function
$\gc$ is regularly varying at 0 with index $\beta_{\mathcal C}\in(0,\infty)$
and the measure $\nuc$ is homogeneous with index
$-\alpha\beta_{\mathcal C}$. For $s\geq1$, define
\begin{align*}
%   \label{eq:def-tc}
  \Tc(s) = \lim_{t\to\infty} \frac{\gc(\bar F_Z(ts))}{\gc(\bar F_Z(t))} = s^{-\alpha\beta_{\mathcal C}} \; .
\end{align*}
Next, Assumption~\ref{hypo:cone} implies that for all $\mathbf u \in
(0,\infty)^h$, it holds that
\begin{align*}
%  \label{eq:def-nuc}
  \lim_{t\to\infty} \frac{\pr(\mathbf u \cdot \zh \in t A)}{\gc(\bar F_Z(t))} =
  \nuc(\mathbf u^{-1} \cdot A) \; .
\end{align*}
This convergence implies that there exists a function $M_A$ such that for all
$\mathbf u \in (0,\infty)^h$,
\begin{align}
  \label{eq:bound}
\sup_{t\geq1}  \frac{\pr(\mathbf u \cdot \zh \in t A)}{\gc(\bar F_Z(t))} \leq M_A(\mathbf u) \; .
\end{align}
Hence, if $\esp[M_A(\bsigma(\xh))]<\infty$, by bounded convergence, we have
\begin{align*}
%  \label{eq:esp-nuc}
  \lim_{t\to\infty} \frac{\pr(\bsigma(\xh) \cdot \zh \in t A)}{\gc(\bar F_Z(t))}
  = \esp[\nuc(\bsigma(\xh)^{-1} \cdot A)] \; .
\end{align*}
For $h=1$, and $A = (1,\infty)$, Potter's bound imply that~(\ref{eq:bound})
holds with $M_A(u) = Cu^{\alpha+\epsilon}$  for some constant $C$, i.e.
\begin{align}
  \label{eq:potter}
  \sup_{t\geq1} \frac{\pr(uZ> t)}{\bar F_Z(t)} \leq C u^{\alpha+\epsilon} \; .
\end{align}

For example, for $m>h$ and $h'\geq0$, and for any Borel measurable set $B
\subset \mathbb{R}^{h'+1}$, we have, by the same bounded convergence argument
\begin{align*}
  % \label{eq:esp-nuc-1}
  \lim_{t\to\infty} \frac{\pr(\mathbf Y_{1,h} \in t A \; , \mathbf Y_{m,m+h'}
    \in B)}{\gc(\bar F_Z(t))} = \esp \left[ \nuc(\bsigma(\xh)^{-1} \cdot A)
    \pr(\mathbf Y_{m,m+h'}\in B \mid \mathcal X) \right] \; .
\end{align*}
If $\esp[\nuc(\bsigma(\xh)^{-1}\cdot A)]>0$ (which in examples is seen
to hold as soon as $\nuc(A)>0$), %\tcr{n'est-ce pas toujours vrai???}
we obtain that the extremogram defined in~(\ref{eq:def-extremogram})
can be expressed as
\begin{align}
  \label{eq:def-rhoA_B_m}
  \rho(A,B,m) & = \lim_{t\to\infty} \pr(\mathbf Y_{m,m+h'} \in B \mid
  \mathbf  Y_{1,h} \in t A)\nonumber \\
  & = \frac{\esp \left[ \nuc(\bsigma(\xh)^{-1}A) \pr(\mathbf
      Y_{m,m+h'}\in B \mid \mathcal X) \right]}
  {\esp[\nuc(\bsigma(\xh)^{-1} \cdot A)]} \; .
\end{align}
We will consider the following type of cones. For $\mathbf j \in
\{0,1\}^h$, let $\cj$ denote the cone defined by
\begin{equation}\label{eq:cones}
  \cj = \{z \in [0 ,\infty]^h \mid \{\sum_{i:j_i=0} z_{j_i}
  \}\prod_{i,j_i=1}z_{j_i} > 0 \} \; .
\end{equation}
In words, a vector $z\in\cj$ if at least one of its entries corresponding to the
components of $\mathbf j$ equal to zero is positive, and all of its entries
corresponding to the components equal to one of $\mathbf j$ are positive.  For
$h=1$, the only cone is $(0,\infty]$ and we will denote it $C_0$ for
consistency of the notation.

A subset $A$ is relatively compact in $\cj$ if and only if there exists
$\eta>0$ such that $\sum_{i:j_i=0} z_{j_i}>\eta$ and $z_{j_i}>\eta$ for all $i$
such that $j_i=1$.

For example, if $h=3$ and $j=(0,0,1)$, then $\cj = ([0,\infty] \times
[0,\infty] \setminus\{(0,0)\}) \times (0,\infty]$, and $A$ is a relatively
compact subset of $\mathcal C_{(0,0,1)}$ if there exists $\epsilon>0$, such that
$(z_1,z_2,z_3) \in A$ implies $z_1>\epsilon$ or $z_2\geq\epsilon$, and
$z_3\geq\epsilon$.

Denote $|\mathbf j|=j_1+\cdots+j_h$, i.e. the number of non zero
components in $\mathbf j$. Then, there exists a non zero Radon
measure $\nu_{\mathbf j}$ on $\cj$ such that for each relatively
compact set $A \in \cj$,
\begin{align*}
  \lim_{t\to\infty} \frac{\pr(\mathbf Z_{1,h} \in tA)}{\bar F_Z(t)^{|\mathbf
      j|+1}} = \nu_{\mathbf j}(A) \; .
\end{align*}
The measure $\nu_{\mathbf j}$ can be described more precisely.
\begin{align*}
  \nu_{\mathbf j}(\mathrm{d}\mathbf z) = \alpha^{|j|+1}\left\{\sum_{i:j_i=0}
    z_{j_i}^{-\alpha-1} \delta_{j_i} (\mathrm{d}z_{j_i}) \right\} \prod_{i:j_i=1}
  z_{j_i}^{-\alpha-1} \mathrm{d}z_{j_i} \; ,
\end{align*}
where $\delta_j$ is Lebesgue's's measure on the $j$-th coordinate axis, i.e. for
any non negative measurable function $\phi$,
$$
\int_{[0,\infty]^h} \phi(z) \delta_j(\mathrm{d}z) = \int_0^\infty
\phi(0,\dots,z_j,\dots,0) \, \mathrm{d}z_j \; .
$$

Moreover, for any relatively compact subset $A$ of $\cj$, and for any
$\epsilon>0$, there exist $\eta>0$ and a constant $C$ (which both depend on
$A$) such that, for all $\mathbf u \in (0,\infty)^h$,
\begin{align}
  \frac{ \pr(\mathbf u \zh \in t A)}{\bar F_Z(u)^{|\mathbf j|+1}} & \leq \frac{
    \pr \left(\cup_{i:j_i=0}\{u_{j_i}Z_{j_i}>\eta\} \cap \cap_{i:j_i=1}
      \{u_{j_i}Z_{j_i}>\eta\}\right) } {\bar F_Z(u)^{|\mathbf j|+1}} \nonumber  \\
  & \leq C \eta^{-(|\mathbf j|+1)(\alpha+\epsilon)} \left\{ \sum_{i:j_i=0}
    (u_{j_i}\vee1)^{\alpha+\epsilon} \right\} \prod_{i:j_i=1}
  (u_{j_i}\vee1)^{\alpha+\epsilon} \; . \label{eq:checkA3-1}
\end{align}
Thus~(\ref{eq:bound}) holds and if
\begin{align}
  \esp \left[ \left\{ \sum_{i:j_i=0} \sigma^{\alpha+\epsilon}(X_{j_i})
    \right\} \prod_{i:j_i=1} \sigma^{\alpha+\epsilon}(X_{j_i}) \right] <
  \infty \; , \label{eq:checkA3-2}
\end{align}
then, cf. (\ref{eq:def-rhoA_B_m}),
\begin{align*}
  \lim_{t\to\infty} \pr(\mathbf Y_{m,m+h'} \in B \mid \mathbf Y_{1,h}
  \in t A) = \frac{\esp \left[ \nu_{\mathbf j}(\bsigma(\xh)^{-1}A) \,
      \pr(\mathbf Y_{m,m+h'}\in B\mid \mathcal X) \right]}
  {\esp[\nu_{\mathbf j}(\bsigma(\xh)^{-1}A)]} \; .
\end{align*}

\begin{remark}
  \label{remark:h_and_m}
We assume that $h<m$. Otherwise, if $m<h$, then vectors $\mathbf
Y_{m,m+h'}$ and $\mathbf Y_{1,h}$ may be asymptotically dependent.
For example, if $\{Z_j\}$
 is i.i.d with the tail distribution as in (\ref{eq:attraction}), then
$\pr(Z_2+Z_3>t \mid Z_1+Z_2>t) \to 1/2$. We do not think that this
is of particular interest, since one is primary interested in
estimating distribution of \textit{future} vector $\mathbf
Y_{m,m+h'}$ based on the \textit{past} observations $\mathbf
Y_{1,h}$.
 \end{remark}

\begin{remark}
  The cones $\mathcal C_{\mathbf j}$ are the only ones such that $\mathbf
  u^{-1}\cdot A \subset \mathcal C$ for all $\mathbf u \in(0,\infty)^h$ and
  every $A\subset \mathcal C$. This assumption can be relaxed and other cones
  could be considered if $\sigma$ is bounded above and away from zero, but this
  is not a desirable assumption since for instance it rules out the case
  $\sigma(x) = \mathrm e^x$.
\end{remark}

\begin{remark}\label{rem:moment-bound}
Consider for example $\sigma(x)=\exp(x)$. Assumption
(\ref{eq:checkA3-2}) is fulfilled for arbitrary (weak and strong)
dependence structure of $\{X_j\}$. The same holds for many moment
assumptions which appear in the paper.
\end{remark}

\subsection{Examples}

\label{sec:xmpl}

%%%%%%%%%%%%%%%%%%%%%%%%%%%%%%%%%%%%%%%%%%%%%%%%%%%%%%%%%
\begin{example}
  \label{xmpl:both}
  Fix some positive integer $h$ and consider the cone $\mathcal C_{\mathbf 1} =
  (0,\infty)^h$. Then~(\ref{eq:reg-var-C}) holds with $g_{h}(t)=t^h$ and
  $\nu_{h}$ defined by
  \begin{align*}
    \nu_{h} (\mathrm{d}z_1,\dots,\mathrm{d}z_h) = \alpha^h \prod_{i=1}^h
    z_i^{-\alpha-1} \mathrm{d}z_i \; .
  \end{align*}
  Consider the set $A$ defined by $A = \{(z_1,\dots,z_h) \in \mathbb{R}_+^h
  \mid z_1>1 , \dots, z_h > 1 \}$.  If
  $$
  \esp \left[ \prod_{i=1}^h\sigma^{\alpha+\epsilon}(X_i) \right] < \infty
  $$
  for some $\epsilon>0$, we obtain, for $m>h$, and $B\in \mathbb{R}^{h'+1}$,
\begin{align*}
  \lim_{t\to\infty} \pr(\mathbf Y_{m,m+h'} \in B \mid Y_1 >t , \dots, Y_h > t) =
  \frac{\esp\left[ \prod_{i=1}^h \sigma^\alpha(X_i) \pr(\mathbf Y_{m,m+h'} \in
      B \mid \mathcal X ) \right] }{\esp \left[ \prod_{i=1}^h
      \sigma^\alpha(X_i) \right] } \; .
\end{align*}
In particular, setting $B=(-\infty,y]$ and $h'=0$, the
limiting conditional distribution of $Y_m$ given that
$Y_1,\dots,Y_h$ are simultaneously large is given by
\begin{align}
  \label{eq:Psi_h}
  \Psi_h(y) = \lim_{t\to\infty} \pr(Y_{m} \leq y \mid Y_1 >t , \dots,
  Y_h > t) = \frac{\esp\left[ \prod_{i=1}^h \sigma^\alpha(X_i)
      F_Z(y/\sigma(X_m)) \right] } {\esp \left[ \prod_{i=1}^h
      \sigma^\alpha(X_i) \right]} \; .
\end{align}

\end{example}
\begin{example}
  \label{xmpl:sum-future}
  Consider again the case $\mathcal C_1 = (0,\infty)$.  Another
  quantity of interest is the limiting distribution of the sum of $h'$
  consecutive values, given that past values are extreme. To keep
  notation simple, consider $h'=1$ and, for $m>1$,
  \begin{align*}
    \Psi^*(y) = \lim_{t\to\infty} \pr(Y_m + Y_{m+1} \leq y \mid Y_1>t) =
    \frac{\esp[\sigma^\alpha(X_1) \pr(Y_m+Y_{m+1} \leq y\mid \mathcal X)]}
    {\esp[\sigma^\alpha(X_1)]} \; .
  \end{align*}
  Estimating this distribution yields for instance empirical quantiles of the
  sum of future returns, given the present one is large.

\end{example}

\begin{example}
  \label{xmpl:sum}
  Consider the cone $\mathcal C_{0,0}= [0,\infty)\times[0,\infty)
  \setminus\{\mathbf 0\}$. Then~(\ref{eq:reg-var-C}) holds with $g_{0,0}(t)=t$
  and $\nu_{0,0}$ defined by
  \begin{align*}
    \nu_{{0,0}} (\mathrm{d}z_1,\mathrm{d}z_2) = \alpha
    \{\delta_{(0,\infty]\times\{0\}}z_1^{-\alpha-1} \mathrm{d}z_1+
    \delta_{\{0\}\times(0,\infty]}z_2^{-\alpha-1} \mathrm{d}z_2\} \; .
  \end{align*}
  The bound~(\ref{eq:bound}) with $M_A(u,v)=C(u^{\alpha+\epsilon} + v^{\alpha+\epsilon})$ for some
  constant $C$.  Consider the set $A$ defined by $A = \{(z_1,z_2) \in
  \mathbb{R}_+^2 \mid z_1+z_2 > 1 \}$.  If
  $\esp[\sigma^{\alpha+\epsilon}(X_1)]<\infty$ for some $\epsilon>0$, we obtain
\begin{align*}
  \lim_{t\to\infty} \pr(\mathbf Y_{m,m+h'} \in B \mid Y_1+Y_2 > t) =
  \frac{\esp\left[ \pr(\mathbf Y_{m,m+h'} \in B \mid \mathcal X )
      (\sigma^\alpha(X_1) + \sigma^\alpha(X_2)) \right] }{\esp[\sigma^\alpha(X_1)]
    +\esp[\sigma^\alpha(X_2)]} \; .
\end{align*}
In particular, take $B=(-\infty,y]$ and $h'=0$. The limiting
conditional distribution of $Y_m$ given $Y_1+Y_2$ is large is defined
by
\begin{align*}
  \Lambda(y) = \lim_{t\to\infty} \pr( Y_m \leq  y  \mid
  Y_1+Y_2 > t) = \frac{\esp[\{\sigma^\alpha(X_1) +
    \sigma^\alpha(X_2)\} F_Z(y/\sigma(X_m)]}
  {\esp[\sigma^\alpha(X_1)+\sigma^\alpha(X_2)]} \; .
\end{align*}
\end{example}

\begin{example}
  \label{xmpl:combine}
  We can combine the previous examples. Consider $A=\{(z_1,z_2,z_3)\in \mathbb{R}_+^3|z_1+z_2>1,z_3>1\}$.
   We may obtain for instance, for $m>3$,
\begin{multline*}
  \lim_{t\to\infty} \pr(\mathbf Y_{m,m+h'} \in B \mid Y_1+Y_2 > t, Y_3 > t) \\
  = \frac{\esp\left[ \pr(\mathbf Y_{m,m+h'} \in B \mid \mathcal X )
      \{\sigma^\alpha(X_1) + \sigma^\alpha(X_2)\} \sigma^\alpha(X_3) \right] }
  {\esp[\{\sigma^\alpha(x_1)+ \sigma^\alpha(X_2)\} \sigma^\alpha(X_3)]} \; ,
\end{multline*}
if $\esp[\{\sigma^{\alpha+\epsilon}(X_1) +
\sigma^{\alpha+\epsilon}(X_2)\} \sigma^{\alpha+\epsilon}(X_3)] <
\infty$ for some $\epsilon>0$.  The relevant cone is $\mathcal
C_{0,0,1}$, $g_{0,0,1}(t) = t^2$ and the associated measure on
$\mathcal C_{0,0,1}$ is defined by
\begin{align*}
  \nu_{0,0,1} = \alpha^2\{\delta_{(0,\infty]\times\{0\}}z_1^{-\alpha-1} \mathrm{d}z_1+
    \delta_{\{0\}\times(0,\infty]}z_2^{-\alpha-1} \mathrm{d}z_2\}z_3^{-\alpha-1} \mathrm{d}z_3 \; .
\end{align*}

\end{example}

\section{Estimation}
\label{sec:estimation}

To simplify the notation, assume that we observe
$Y_1,\dots,Y_{n+m+h'}$. An estimator $\hat\rho_n(A,B,m)$ is naturally
defined by
\begin{align*}
  \hat\rho_n(A,B,m) = \frac{ \sum_{j=1}^n \mathbf{1}_{ \{ \mathbf Y_{j,j+h-1} \in
      Y_{(n:n-k)} A \}} \mathbf{1}_{\{ \mathbf Y_{j+m,j+m+h'} \in B\}} }
  {\sum_{j=1}^n \mathbf{1}_{ \{ \mathbf Y_{j,j+h-1} \in Y_{(n:n-k)} A \}}} \; ,
\end{align*}
where $k$ is a user chosen threshold and $Y_{(n:1)} \leq \dots \leq Y_{(n:n)}$
are the increasing order statistics of the observations $Y_1,\dots,Y_n$.
We will also consider the case $B=(-\boldsymbol\infty,\mathbf y]$, i.e. the
case of the limiting conditional distribution of $\mathbf Y_{m,m+h'}$ given
$\mathbf Y_{1,h} \in t A$, i.e.
\begin{align}
  \nonumber \Psi_{A,m,h'}(\mathbf y) & = \lim_{t\to\infty} \pr(\mathbf Y_{m,m+h'}
  \leq \mathbf y \mid \mathbf Y_{1,h} \in t A) \\
  & = \rho(A,(\boldsymbol\infty,\mathbf y],m) = \frac{\esp[\nu_{\mathcal
      C}(\boldsymbol\sigma(\xh)^{-1} \cdot A) \prod_{i=1}^{h'}
    F(y_i/\sigma(X_{m+i}))]} {\esp[\nu_{\mathcal C} (\boldsymbol\sigma(\xh)^{-1}
    \cdot A)]}
\label{eq:def-Psi}\; .
\end{align}
An estimator $\hat\Psi_{n,A,m,h'}$ of $\Psi_{A,m,h'}$ is defined on $\mathbb{R}^{h'+1}$ by
\begin{align}
  \label{eq:def-hatPsi}
  \hat\Psi_{n,A,m,h'}(\mathbf y) = \frac{\sum_{j=1}^n \mathbf{1}_{ \{ \mathbf
      Y_{j,j+h-1} \in Y_{(n:n-k)} A \}} \mathbf{1}_{\{ \mathbf Y_{j+m,j+m+h'} \leq
      \mathbf y\}}} {\sum_{j=1}^n \mathbf{1}_{ \{ \mathbf Y_{j,j+h-1} \in Y_{(n:n-k)}
        A \}}} \; .
\end{align}

In order to obtain statistical results, we need additional assumptions. We first
state two assumptions which will be needed to prove the weak convergence of a
multivariate conditional empirical process.

  \begin{hypothesis}
    \label{hypo:covmdep}
    For $j=1,\dots,h$, there exist functions $\covmdep_{j}$ such that for all
    $s,s'\geq1$, $\mathbf u,\mathbf v \in (0,\infty)^h$,
    \begin{align}
     \label{eq:limcovmdep}
      \lim_{t\to\infty} \frac{\pr(\mathbf u \cdot \zh \in t sA , \mathbf v
        \cdot \mathbf Z_{j,j+h-1} \in t s'A)} {g_{\mathcal C}(\bar F_Z(t))} = \covmdep_j
      (A,\mathbf u,\mathbf v, s,s') \; .
    \end{align}
  \end{hypothesis}
  For $j=1$ we only need that~(\ref{eq:limcovmdep}) holds with $\mathbf
  u=\mathbf v$.  If $A$ is a cone, then~(\ref{eq:limcovmdep}) holds for $j=1$
  with $\covmdep_1(A,\mathbf u,\mathbf u,s,s') = \Tc(s\vee s') \nuc(\mathbf
  u^{-1} \cdot A)$ as an immediate consequence of Assumption~\ref{hypo:cone}.
  It may happen that $\covmdep_j(A,\cdot) \equiv 0$ for
  $j=2,\dots,h$. Intuitively, this happens if $\mathbf u \cdot \zh$ and $\mathbf
  v \cdot \mathbf Z_{j,j+h-1}$ belong simultaneously to $t A$ implies that at
  least $h+1$ coordinates of $\mathbf Z_{1,h+j-1}$ are large. This is the case
  for instance for Examples~\ref{xmpl:both} and~\ref{xmpl:combine}.  Actually,
  Assumption~\ref{hypo:covmdep} holds for the cones $\cj$, but a precise
  description of the functions $\covmdep_j$ when they are not identically zero
  would be extremely involved.  This will only be done for
  Example~\ref{xmpl:sum}. See Section~\ref{sec:xmpl-cntnd}.

By Cauchy-Schwartz inequality, if
  Assumptions~\ref{hypo:cone} and~\ref{hypo:covmdep} hold, then, for $s,s'\geq1$,
\begin{align*}
  \covmdep_j(\mathbf u,\mathbf v, sA,s'A) \leq \sqrt{M_A(\mathbf u)
    M_{A}(\mathbf v)} \; .
\end{align*}
Thus, if $\esp[M_A(\bsigma(\xh))]<\infty$, then the convergence
in~(\ref{eq:limcovmdep}) is also in $L^1(\bsigma(\xh),\bsigma(\mathbf
X_{j,j+h-1}))$.

The next assumption is needed for the quantities (that will appear in the
limiting distributions) to be well defined and to use bounded convergence
arguments.
\begin{hypothesis}
  \label{hypo:moments-general}
  $\esp[M_A^2(\bsigma(\xh))] < \infty$.
\end{hypothesis}

As usual, the bias of the estimators will be bounded by a second order type
condition.  Let $k$ be a non decreasing sequence of integers, let $F_Y$ denote
the distribution of $Y$ and let $u_n = (1/\bar F_Y)^\leftarrow(n/k)$.  Consider
the measure defined on the Borel subsets of $\mathcal C$ by
\begin{align}
  \label{eq:def-muc}
  \muc(A) = \frac{\esp[\nu_{\mathcal C}(\bsigma(\xh)^{-1}\cdot A)]}
  {(\esp[\sigma^\alpha(X)])^{\beta_{\mathcal C}}} \; .
\end{align}
We introduce a rate of convergence:
\begin{align}
   \label{eq:vitesse}
   v_n(A) = \esp \left[ \sup_{s\geq1} \left| \frac{\pr(\mathbf Y_{1,h} \in u_n s
         A \mid \mathcal X)}{g_{\mathcal C}(k/n)} - \Tc(s) \muc(A) \right|
   \right] \; .
\end{align}
\begin{lemma}
  \label{lem:uniform}
  Under Assumption~\ref{hypo:cone} and~\ref{hypo:moments-general},
  $\lim_{n\to\infty} v_n(A) = 0$.
\end{lemma}

We need also the following quantities, which are well defined
under Assumptions~\ref{hypo:cone}, \ref{hypo:covmdep}
and~\ref{hypo:moments-general}.
For $j=2,\dots,h$ and measurable subsets $B,B'$ of~$\mathbb{R}^{h'+1}$, define
\begin{multline}
  \label{eq:def-rj}
  \mathcal R_j(A,B,B') \\
  = \frac{\esp \big[ \covmdep (A,\bsigma(\xh),\bsigma(\mathbf X_{j,j+h-1}),0,0)
    \times \pr(\mathbf Y_{m,m+h'} \in B, \mathbf Y_{m+j-1,m+h'+j-1} \in B' \mid
    \mathcal X) \big]} {(\esp[\sigma^\alpha(X)])^{-\beta_{\mathcal C}} \,
    \esp[\nuc(\bsigma(\xh)^{-1}\cdot A)]}  \\
  + \frac{\esp \big[ \covmdep (A,\bsigma(\xh),\bsigma(\mathbf X_{j,j+h-1}),0,0)
    \times \pr(\mathbf Y_{m,m+h'} \in B', \mathbf Y_{m+j-1,m+h'+j-1} \in B \mid
    \mathcal X) \big]} {(\esp[\sigma^\alpha(X)])^{-\beta_{\mathcal C}} \,
    \esp[\nuc(\bsigma(\xh)^{-1}\cdot A)]} \; .
\end{multline}
For brevity, denote $\mathcal R_j(A,B) = \mathcal R_j(A,B,B)$.

\subsection{General result: weak dependence}

We can now state our main result in the weak dependence setting, i.e. when
absolute summability~(\ref{eq:weakdep}) of the autocovariance function of the
process $\{X_j\}$ holds.

In order to simplify the proof, we make an additional assumption.
\begin{hypothesis}
  \label{hypo:simplification}
  If $s<t$ then $tA \subset sA$.
\end{hypothesis}
This assumptions holds for all the examples considered here and most common
examples.

\begin{theorem}
  \label{theo:estimation-general}
  Let Assumptions~\ref{hypo:cone}, \ref{hypo:covmdep},
  \ref{hypo:moments-general}, \ref{hypo:simplification} and the weak
  dependence condition~(\ref{eq:weakdep}) hold. Assume moreover that
  $\muc(A)>0$, $k/n\to0$, $ng_{\mathcal C}(k/n)\to\infty$ and
  \begin{align}
     \label{eq:second-ordre-adhoc}
     \lim_{n\to\infty} n g_{\mathcal C}(k/n)  \; v_n(A) = 0 \; .
  \end{align}

  Then
$$
\sqrt{ng_{\mathcal C}(k/n) \muc(A)} \{\hat \rho_n(A,B,m) - \rho(A,B,m)\}
$$
converges weakly to a centered Gaussian distribution with variance
\begin{multline}
  \label{eq:cov-correct}
  \rho(A,B,m) \{1 - \rho(A,B,m)\} \\
  + \sum_{j=2}^{h\wedge(m-h)} \big\{\mathcal R_j(A,B) -2 \rho(A,B,m) \mathcal
  R_j(A,B,\mathbb{R}^{h'+1}) + \rho^2(A,B,m) \mathcal R_j(A,\mathbb{R}^{h'+1}) \big\} \;
  .
 \end{multline}
\end{theorem}

\begin{remark}
  If $h=1$ or if the functions $\covmdep_j$ defined in
  Assumption~\ref{hypo:covmdep} are identically zero for $j\geq2$, then the
  limiting covariance in (\ref{eq:cov-correct}) is simply
  $\rho(A,B,m)\{1-\rho(A,B,m)\}$.

  Otherwise, the additional terms can be canceled by modifying the estimator
  of $\hat\rho_n(A,B,m)$. Assuming we have $nh+m+h'+1$ observations, we can
  define
    \begin{align*}
      \tilde \rho_n(A,B,m) = \frac{\sum_{j=1}^n \mathbf{1}_{\{\mathbf Y_{(j-1)h+1,jh}
          \in Y_{(n:n-k)} A\}} \mathbf{1}_{\{\mathbf Y_{(j-1)h+m,(j-1)h+m+h'} \in
          B\}}} {\sum_{j=1}^n \mathbf{1}_{\{\mathbf Y_{(j-1)h+1,jh} \in Y_{(n:n-k)}
          A\}}}
    \end{align*}
    Noting that the events $\{\mathbf Y_{j,j+h-1} \in A\}$ are $h$-dependent
    conditionally on $\mathcal X$, the proof of
    Theorem~\ref{theo:estimation-general} can be easily adapted to show that
    the limiting variance of $\sqrt{ng_{\mathcal C}(k/n)}\{\tilde \rho_n(A,B,m)
    - \rho(A,B,m)\}$ is the same as in the case where $\covmdep_j\equiv0$ for
    $j=2,\dots,h$. But this is of course at the cost of an increase of the
    asymptotic variance, due to a different sample size.
\end{remark}

We can also obtain the functional convergence of the estimator
$\hat\Psi_{n,A,m,h'}$ of the limiting conditional distribution function
$\Psi_{A,m,h'}$, defined respectively in~(\ref{eq:def-hatPsi})
and~(\ref{eq:def-Psi}).
\begin{corollary}
  \label{coro:empirical}
  Under the Assumptions of Theorem~\ref{theo:estimation-general}, and
  if moreover the distribution $\Psi_{A,m,h'}$ is continuous, then
  $$
  \sqrt{n g_{\mathcal C}(k/n) \muc(A)} \{\hat \Psi_{n,A,m,h'} - \Psi_{A,m,h'}\}
$$
converges in $\mathcal D(\mathbb{R}^{h'+1})$ to a Gaussian process. If $h=1$ or if the
functions $\covmdep_j$ are identically zero for $j=2,\dots,h$, then the limiting
process can be expressed as $\mathbb B \circ\Psi_{A,m,h'}$, where $\mathbb B$ is
the standard Brownian bridge.
\end{corollary}

Note that a sufficient condition for $\Psi_{A,m,h'}$ to be continuous is that $F_Z$ is continuous.

\subsection{General result: long memory}
\label{sec:long-memory}
We now state our results in the framework of long memory.
This requires several additional notions, such as multivariate
Hermite expansion and Hermite ranks which are recalled in Appendix \ref{sec:hermite}.

Define the functions $G_n$ and $G$ for $(\mathbf x,\mathbf x')
\in\mathbb{R}^{h}\times\mathbb{R}^{h'+1}$ and $s\geq1$ by
\begin{align}
  G_n(A,B,s,\mathbf x,\mathbf x') & = \frac{\pr(\boldsymbol \sigma(\mathbf x)
    \cdot \mathbf Z_{1,h} \in u_ns A)}{g(k/n)} \; \pr(\bsigma(\mathbf x')
  \cdot \mathbf Z_{m,m+h'}\in B)
\end{align}
\begin{align}
  G(A,B,\mathbf x,\mathbf x') & = \lim_{n\to\infty} G_n(A,B,1,\mathbf x,\mathbf
  x')  \nonumber \\
& = \frac{(\nuc(\bsigma(\mathbf x)^{-1}\cdot A)}
  {\esp[\sigma^\alpha(X_1)])^{\beta_{\mathcal C}}} \; \pr(\bsigma(\mathbf x')
  \cdot \mathbf Z_{m,m+h'}\in B) \; .
\end{align}
Let $\tau_n(A,B,s)$ and $\tau(A,B)$ be the Hermite ranks with respect
to $(\mathbf X_{1,h},\mathbf X_{m,m+h'})$ of the functions
$G_n(A,B,s,\cdot,\cdot)$ and $G(A,B,\cdot,\cdot)$, respectively.  Define
$\tau(A) = \tau(A,\mathbb{R}^d)$.
\begin{hypothesis}
  \label{hypo:hermite-rank}
  For large $n$, $\inf_s \tau_n(A,B,s) = \tau(A,B)$ and $\tau(A,B) \leq \tau(A)$.
\end{hypothesis}
This assumption is fulfilled for example when $\sigma(x)=\exp(x)$, in which case all the considered Hermite ranks are equal to one, or if $\sigma$ is an even function with Hermite rank 2 (such as $\sigma(x)=x^2$), in which case they are equal to two.
The modification of Theorem \ref{theo:estimation-general} reads as follows.

\begin{theorem}
  \label{theo:estimation-general-Thlrd}
  Assume that $\{X_j\}$ is the long memory Gaussian sequence with
  covariance given by (\ref{eq:lrd}).  Let
  Assumptions~\ref{hypo:cone},~\ref{hypo:covmdep},
  \ref{hypo:moments-general}, \ref{hypo:simplification} and \ref{hypo:hermite-rank} hold,
  $\muc(A)>0$ and $k/n\to0$, $ng_{\mathcal C}(k/n)\to\infty$ and
  \begin{align}
     \label{eq:second-ordre-adhoc-Thlrd}
     \lim_{n\to\infty} \left\{n g_{\mathcal C}(k/n) \wedge
       \gamma_n^{-\tau(A,B)/2} \right\} \; v_n(A) = 0 \; .
  \end{align}

  \begin{enumerate}[(i)]
  \item \label{item:iid-Thlrd} If $ng_{\mathcal C}(k/n) \gamma_n^{\tau(A,B)} \to 0$,
    then $$\sqrt{ng_{\mathcal C}(k/n)
      \muc(A)} \{\hat \rho_n(A,B,m) - \rho(A,B,m)\}$$ converges to a centered
    Gaussian distribution with variance given in (\ref{eq:cov-correct})
  \item \label{item:Thlrd} If $ng_{\mathcal C}(k/n) \gamma_n^{\tau(A,B)} \to
    \infty$, then $\gamma_n^{-\tau(A,B)/2} \{\hat \rho_n(A,B,m) - \rho(A,B,m)\}$
    converges weakly to a distribution which is non-Gaussian except if
    $\tau(A,B)=1$.
\end{enumerate}
\end{theorem}

The exact definition of the limiting distribution will be given in Section
\ref{sec:proof}. It suffices to mention here that this distribution depends on
$H$ and $\tau(A,B)$. The meaning of the above result is the following. In the
long memory setting, it is still possible to obtain the same limit as in the
weakly dependent case, if $k$ (i.e., the number of high order statistics used in
the definition of the estimators) is not too large, so that both the bias and
the long memory effect are canceled.

Define a new Hermite rank $\tau^*(A) =
\inf_{y\in\mathbb{R}^{h'+1}} \tau(A,(\boldsymbol\infty, \mathbf y])$.
\begin{corollary}
  \label{coro:empirical-lrd}
  Under the Assumptions of Theorem~\ref{theo:estimation-general-Thlrd}, if the
  distribution function $\Psi_{A,m,h'}$ is continuous and if $\tau^*(A) \leq
  \tau(A)$, then
  \begin{itemize}
  \item If $n g_{\mathcal C}(k/n) \gamma_n^{\tau^*(A)} \to 0$, then
$$
\sqrt{n g_{\mathcal C}(k/n) \muc(A)} \{\hat    \Psi_{n,A,m,h'} - \Psi_{A,m,h'}\}
$$
converges in $\mathcal D((-\infty,+\infty)^{h'+1}$ to a Gaussian process. If $h=1$ or if the functions
    $\covmdep_j$ are identically zero for $j=2,\dots,h$, then the limiting
    process can be expressed as $\mathbb B \circ\Psi_{A,m,h'}$, where $\mathbb B$ is the
    standard Brownian bridge.
  \item If $n g_{\mathcal C}(k/n) \gamma_n^{\tau^*(A)} \to \infty$, then
    $\gamma_n^{-\tau^*(A)/2} \{\hat \Psi_{n,A,m,h'} - \Psi_{A,m,h'}\}$ converges in $\mathcal D((-\infty,+\infty)^{h'+1}$
     to a process which can be expressed as $J_{A,m,h'} \cdot \aleph$ where $J_{A,m,h'}$ is a deterministic function and $\aleph$ is
    a random variable, which is non Gaussian except if $\tau^*(A)=1$.
  \end{itemize}
\end{corollary}
The exact definition of the function $J_{A,m,h'}$ and of the random variable
$\aleph$ will be given in Section \ref{sec:proof}. Anyhow, they are not of much
practical interest. In practice, the main goal will be to choose the number $k$
of order statistics used in the estimation procedure so that both the bias and
the long memory effect are canceled, and the limiting distribution of the
weakly dependent case can be used in the inference.

\subsection{Examples}
\label{sec:xmpl-cntnd}

We now discuss the Examples introduced in Section~\ref{sec:xmpl}.
In order to evaluate the rate of convergence~(\ref{eq:vitesse}), it is necessary
to introduce a second order regular variation condition.  We follow here
\cite{drees:1998}.

\begin{hypothesis}
  \label{hypo:2RV-Z}
There exists a bounded non increasing
function $\eta^*$ on $[0,\infty)$, regularly varying at infinity with
index $-\alpha\zeta$ for some $\zeta\geq0$, and such that
$\lim_{t\to\infty} \eta^*(t)=0$ and there exists a measurable function
$\eta$ such that for $z>0$,
\begin{gather*}
  \pr(Z>z) = c z^{-\alpha} \exp\left(\int_1^z \frac{\eta(s)}s \, \mathrm{d}s\right) \; ,
%  \label{eq:representation}
  \\
  \exists C>0 \;, \ \ \forall s \geq 0 \; , \ \ |\eta(s)| \leq
  C \eta^*(s) \; .
%  \label{eq:borne-eta}
\end{gather*}
\end{hypothesis}
On account of Breiman's lemma, if the tail of $Z$ is regularly
varying with index $-\alpha$, then the same holds for
$Y=\sigma(X)Z$, as long as $X$ and $Z$ are independent, and
$\esp[\sigma^{\alpha}(X)]<\infty$. Also, (SO) property is
transferred from the tail of $Z$ to $Y$; See \cite[ Proposition
2.1]{kulik:soulier:2011}.

For the sake of simplicity and clarity of exposition, we will make in this
section the usual assumption that $\sigma(x)=\exp(x)$, so that the Hermite rank
of $\sigma$ is 1. This will avoid to define many auxiliary functions and
Hermite ranks. But the examples can of course be treated in a more general
framework. Also, we will only state the convergence results under the conditions
which imply that the limiting distribution is the same as in the weak dependence
case, since this is the case of practical interest. We only treat
Examples~\ref{xmpl:both} and~\ref{xmpl:sum} since they exhibit the two different
possibility for the limiting distributions. The computations for the other
examples are straightforward.
\subsubsection{Example~\ref{xmpl:both} continued}
Fix integers $h\geq 1$ and $m>h$.  Recall the formula (\ref{eq:Psi_h})
for the conditional distribution of $Y_m$ given that $Y_1,\dots,Y_h$
are simultaneously large. Its estimator $\hat\Psi_{n,h}$ is defined by
\begin{align*}
  \hat\Psi_{n,h}(y) = \frac{ \sum_{j=1}^n \mathbf{1}_{\{Y_j > Y_{(n:n-k)},
      \dots,Y_{j+h-1} > Y_{(n:n-k)},Y_{j+m} \leq y\}}} { \sum_{j=1}^n
    \mathbf{1}_{\{Y_{j} > Y_{(n:n-k)}, \dots,Y_{j+h-1} >
      Y_{(n:n-k)}\}}} \;
\end{align*}
with a user chosen $k$.

Assumption~\ref{hypo:covmdep} holds with $\covmdep_j(A,\cdot) \equiv 0$,
$j=2,\dots,h$. Assumption~\ref{hypo:2RV-Z} and
\cite[Proposition~2.8]{kulik:soulier:2011} imply that if moreover
\begin{align}
  \label{eq:moment-2RV-Z}
  \esp \left[ \prod_{i=1}^h\sigma^{2\alpha(\zeta+1)+\epsilon}(X_i) \right] < \infty
\end{align}
for some $\zeta,\epsilon>0$,
a bound for $v_n(A)$ is then given by
\begin{align}
  \label{eq:vitesse-second-ordre}
  v_n(A) = O(\eta^*(u_n)) \; .
\end{align}
The moment restriction (\ref{eq:moment-2RV-Z}) is quite weak. In particular, it
is fulfilled for $\sigma(x)=\exp(x)$; see Remark \ref{rem:moment-bound}. Recall
that in this example Assumption \ref{hypo:cone} and \ref{hypo:covmdep} hold and
the functions $\mathcal L_j$ therein are vanishing for $j\geq2$. Also,
Assumption \ref{hypo:moments-general} is implied by \eqref{eq:moment-2RV-Z}.
\begin{corollary}
  \label{coro:conditional-distribution-both}
  Assume that $\sigma(x)=\exp(x)$.  Let
  Assumption~\ref{hypo:2RV-Z} and (\ref{eq:moment-2RV-Z}) hold. Let $k$ be such
  that $k/n\to0$, $n(k/n)^h\to\infty$, and
  \begin{align}
    \label{eq:second-ordre-h}
    \lim_{n\to\infty} (n(k/n)^h)^{1/2}  \eta^*(u_n) = 0 \; .
  \end{align}
  In the weakly dependent case \eqref{eq:weakdep} or in the long memory case
  \eqref{eq:lrd} if moreover $n(k/n)^h \gamma_n \to 0$, then
\begin{align*}
  \sqrt{n(k/n)^h} (\hat\Psi_{n,h}-\Psi_h) \Rightarrow
  \left(\frac{\esp[\sigma^\alpha(X_1) \cdots
      \sigma^\alpha(X_h)]}{\esp^h[\sigma^\alpha(X_1)]}\right)^{-1/2} \, \mathbb B \circ \Psi_h
\end{align*}
weakly in $\mathcal D((-\infty,\infty))$, where $\mathbb B$ is the
standard Brownian bridge.
\end{corollary}

% \subsubsection{Example~\ref{xmpl:sum-future} continued}
% An estimator of the limiting distribution $\Psi^*$
% of the sum of $Y_m+Y_{m+1}$ given that $Y_1$ is large is given by
% \begin{align*}
%   \hat\Psi_n^*(y) = \frac1k \sum_{j=1}^n \mathbf{1}_{\{Y_j>Y_{(n:n-k)};Y_{j+m-1}+ Y_{j+m} \leq y\}} \; .
% \end{align*}
% Since the conditioning event is one-dimensional, Assumption~\ref{hypo:covmdep}
% is irrelevant.   The moment bound
% \begin{align}
%   \label{eq:moment-2RV-Z-sum-future}
%   \esp[\sigma^{2\alpha(\zeta+1)+\epsilon}(X)] < \infty
% \end{align}
% for some $\epsilon>0$ yields the bound~(\ref{eq:vitesse-second-ordre}) for the
% rate of convergence.

% \begin{theorem}
%   \label{theo:conditional-distribution-sum-future}
%   Let the assumptions of Theorem~\ref{theo:estimation-general} hold.  Let
%   Assumption~\ref{hypo:2RV-Z} and~(\ref{eq:moment-2RV-Z-sum-future}) hold. Let
%   $k$ be such that $k\to\infty$ and $k/n\to0$, and
%   \begin{align}
%     \label{eq:second-ordre-sum-future}
%     \lim_{n\to\infty} k^{1/2} \eta^*(u_n) = 0 \; .
%   \end{align}
%   Then in the weakly dependent case then $\sqrt{k} (\hat\Psi_{n}^*-\Psi^*)$
%   converges in $\mathcal D((-\infty,\infty))$ to $\mathbb B \circ \Psi^*$, where
%   $\mathbb B$ is the standard Brownian bridge.
% \end{theorem}

\subsubsection{Example~\ref{xmpl:sum} continued}

Consider the estimation of
\begin{align*}
  \Lambda(y) = \lim_{t\to\infty} \pr(Y_m \leq y \mid Y_1+Y_2>t) =
  \frac{\esp[\{\sigma^\alpha(X_1)+\sigma^\alpha(X_2)\}F_Z(y/\sigma(X_m))]}
  {\esp[\sigma^\alpha(X_1)+\sigma^\alpha(X_2)]}\; .
\end{align*}

An estimator if defined by
\begin{align*}
  \hat\Lambda_{n}(y) = \frac{\sum_{j=1}^n \mathbf{1}_{\{Y_j+Y_{j+1}>Y_{(n:n-k)}\}}
    \mathbf{1}_{\{Y_{j+m} \leq y \}}} {\sum_{j=1}^n \mathbf{1}_{\{Y_j+Y_{j+1}>Y_{(n:n-k)}\}}} \; .
\end{align*}

We have already shown that Assumption~\ref{hypo:cone} holds and
Assumption~\ref{hypo:simplification} holds trivially.
Assumption~\ref{hypo:covmdep} holds with the function $\covmdep_2$ defined by
\begin{align}
  \covmdep_2(A,u_1,u_2,v_1,v_2,s,s') = \left( \frac{1+s}{u_2} \vee
    \frac{1+s'}{v_1} \right) ^{-\alpha} \; .
\end{align}
If $\esp[\sigma^{2\alpha(\zeta+1)+\epsilon}(X_1)]<\infty$, then
Assumption~\ref{hypo:moments-general} holds and applying
Lemma~\ref{lem:second-ordre-convol-weight}, we obtain a bound for $v_n(A)$:
\begin{align}
  \label{eq:vitesse-sum}
  v_n(A) = O \left(\eta^*(u_n) + u_n^{-1} \int_0^{u_n} \bar F_Z(s) \, \mathrm{d}s\right)
  \; .
\end{align}
 as soon as
\begin{align}
  \label{eq:hypo3-check}
\end{align}
\begin{corollary}
  \label{theo:conditional-distribution-given-sum}
  Let Assumption~\ref{hypo:2RV-Z} and~(\ref{eq:moment-2RV-Z}) hold. Let $k$ be
  such that $k\to\infty$, $k/n\to0$ and
  \begin{align*}
    \lim_{n\to\infty} k^{1/2}  \left( \eta^*(u_n) +
      u_n^{-1} \int_0^{u_n} \bar F_Z(s) \, \mathrm{d}s\right) = 0 \; .
  \end{align*}
  In the weakly dependent case~(\ref{eq:weakdep}) or in the long memory
  case~(\ref{eq:lrd}) if moreover $k \gamma_n \to 0$, then
$$
k^{1/2} (\hat\Lambda_{n}-\Lambda) \Rightarrow \left(
  \frac{\esp[\sigma^\alpha(X_1)+\sigma^\alpha(X_2)]}
  {\esp[\sigma^\alpha(X_1)]} \right)^{-1/2} \; \mathbb W
$$
weakly in $\mathcal D((-\infty,\infty))$, where $\mathbb W$ is a
Gaussian process with covariance
  \begin{multline*}
    \cov(\mathbb W(y), \mathbb W(y')) = \Lambda(y \wedge y') - 2 \Lambda(y)
    \Lambda(y')    \\
    + \frac{\esp[\sigma^\alpha(X_2) \{F_Z(y/\sigma(X_m))F_Z(y'/\sigma(X_{m+1}))
      + F_Z(y/\sigma(X_m))F_Z(y'/\sigma(X_{m+1}))\}]} {\esp[\sigma^\alpha(X_1) +
      \sigma^\alpha(X_2)]} \; .
  \end{multline*}

\end{corollary}

\begin{remark}
    If the estimator if modified by taking only every other observation, then
    $\sqrt k(\hat\Lambda_{n}-\Lambda)$ converges weakly to
    $2  \mathbb B \circ \Lambda$  where $\mathbb B$ is the standard Brownian bridge.

  \end{remark}

\section{Proofs}
\label{sec:proof}

For clarity of notation, denote $\sigma_i = \sigma(X_i)$,
$g=g_{\mathcal C}$, $T=T_{\mathcal C}$ and $\beta=\beta_{\mathcal
  C}$. Recall that $F_Y$ denotes the distribution function of $Y$ and
$u_n = (1/\bar F_Y)^\leftarrow (n/k)$. By~(\ref{eq:attraction}) and
the regular variation of $g$, it holds that $\bar F_Y(u_n) \sim
\esp[\sigma_0^\alpha] \bar F_Z(u_n)$ and
\begin{align*}
  \lim_{n\to\infty} \frac{g(k/n)}{g(\bar F_Z(u_n))} =
  (\esp[\sigma_0^\alpha])^\beta \; .
\end{align*}
Whenever there is no risk of confusion, we omit dependence on $h$,
$m$, $h'$ and $A$ in the notation.  For $j=1,\dots,n$, define the
following random variables
\begin{align}
  \label{eq:def_of_W_and_V}
  W_{j,n}(s) & = \mathbf{1}_{\{\mathbf Y_{j,j+h-1} \in u_nsA\}} \; , s\geq 1 \; , \ \
  \ V_{j}(B) = \mathbf{1}_{\{\mathbf Y_{j+m,j+m+h'}\in B\}} \; .
\end{align}
Assumption \ref{hypo:cone} together with the choice of $u_n$ implies that
(recall the definitions (\ref{eq:def-rhoA_B_m}) and (\ref{eq:def-muc}) of
$\rho(A,B,m)$ and $\muc(A)$),
\begin{align}
  \label{eq:asymp_of_W}
  \lim_{n\to\infty}\frac{\esp[W_{j,n}(s)]}{ g(k/n)} &  = T(s)\muc(A) \; , \\
  \lim_{n\to\infty}\frac{\esp[W_{j,n}(s)V_{j}(B)]}{ g(k/n)} & = T(s)
  \muc(A) \rho(A,B,m) \; .  \label{eq:asymp_of_V}
\end{align}
Define, for $s\geq1$ and $\mathbf x\in\mathbb{R}^h$ and $\mathbf x'\in\mathbb{R}^{h'+1}$,
the functions $L_n$ and $G_n$ by
\begin{align}
  \label{eq:function-L}
  L_n(s, \mathbf x) & = \frac{\pr(\boldsymbol \sigma(\mathbf x) \cdot \mathbf
    Z_{1,h} \in u_ns A)}{g(k/n)} \;, \\
  G_n(s,\mathbf x,\mathbf x',B) & = L_n(s,\mathbf x) \; \pr(\bsigma(\mathbf x')
  \cdot \mathbf Z_{m,m+h'}\in B) \; .\label{eq:def-gn-general}
\end{align}
%(The probability above becomes the product since it is assumed that $h<m$; see Remark \ref{remark:h_and_m}).
%Also, we suppress dependence on $A$ and $B$.
With these notations, we have,
\begin{align*}
  L_n(s,\mathbf X_{j,j+h-1}) & = \frac{\esp[W_{j,n}(s)\mid \mathcal X]}{g(k/n)}  \; ,  \\
  G_n(s,\mathbf X_{j,j+h-1},\mathbf X_{j+m,j+m+h'},B) & =
  \frac{\esp[W_{j,n}(s)V_j(B)\mid \mathcal X]}{g(k/n)} \; .
\end{align*}
For $\mathbf x\in\mathbb{R}^h$, denote
\begin{gather}
  \label{eq:def-L}
  L(\mathbf x) = \frac{\nuc(\bsigma(\mathbf x)^{-1}\cdot
    A)}{(\esp[\sigma^{\alpha}(X)])^{\beta}} \; ,
\end{gather}
so that $\esp[L(\mathbf X_{1,h})]=\muc(A)$.
\begin{proof}[Proof of Lemma~\ref{lem:uniform}]
Write
\begin{multline*}
  L_n(s,\mathbf x) - \Tc(s) L(\mathbf x) \\
  = \left \{ \frac{g(\bar F_Z(u_ns))} {g(k/n)} -
    (\esp[\sigma^\alpha(X)])^{-\beta} \, \Tc(s) \right\} \frac{ \pr(\mathbf
    \sigma(\mathbf x) \cdot \mathbf Z_{1,h} \in u_n s A) } {g(\bar F_Z(u_ns))}
  \\
   + (\esp[\sigma^\alpha(X)])^{-\beta} \Tc(s) \left\{ \frac{\pr(\mathbf
      \sigma(\mathbf x) \cdot \mathbf Z_{1,h} \in u_n s A) }{g(\bar
      F_Z(u_ns))}- (\esp[\sigma^\alpha(X)])^\beta L(\mathbf x) \right\} \; .
\end{multline*}
Thus, recalling the definition of $v_n$ from~(\ref{eq:vitesse}), we have
\begin{align*}
  v_n(A) & \leq \sup_{s\geq1} \left| \frac{g(\bar F_Z(u_ns))} {g(k/n)} -
    (\esp[\sigma^\alpha(X)])^{-\beta} \, \Tc(s) \right|
  \esp[M_A(\boldsymbol\sigma(\mathbf X_{1,h}))] \\
  & + (\esp[\sigma^\alpha(X)])^{-\beta} \esp \left[ \sup_{s\geq0} \left | \frac{\pr(\mathbf
        \sigma(\xh) \cdot \mathbf Z_{1,h} \in u_ns A \mid \mathcal X) }{g(\bar F_Z(u_ns))} -
      (\esp[\sigma^\alpha(X)])^\beta L(\xh) \right| \right]
\end{align*}
By Assumption~\ref{hypo:cone}, for all $\mathbf x\in\mathbb{R}^h$,
\begin{align*}
  \lim_{n\to\infty} \sup_{s\geq1} \left | \frac{\pr(\mathbf \sigma(\mathbf x)
      \cdot \mathbf Z_{1,h} \in u_ns A) }{g(\bar F_Z(u_ns))} -
    (\esp[\sigma^\alpha(X)])^\beta L(\mathbf x) \right| = 0 \; .
\end{align*}
Moreover, by~(\ref{eq:bound}),
\begin{align*}
  \sup_{s\geq1} \left | \frac{\pr(\mathbf \sigma(\mathbf x) \cdot \mathbf
      Z_{1,h} \in u_ns A) }{g(\bar F_Z(u_ns))} -
    (\esp[\sigma^\alpha(X)])^\beta L(\mathbf x) \right| \leq 2M_A(
  \boldsymbol\sigma(\mathbf x)) \; .
\end{align*}
Thus, by Assumption~\ref{hypo:moments-general} and bounded convergence,
\begin{align*}
  \lim_{n\to\infty} \esp \left[ \sup_{s\geq0} \left | \frac{\pr(\mathbf
        \sigma(\xh) \cdot \mathbf Z_{1,h} \in u_ns A \mid \mathcal X) }{g(\bar F_Z(u_ns))} -
      (\esp[\sigma^\alpha(X)])^\beta L(\xh) \right|^2 \right] = 0 \; .
\end{align*}
Since $g \circ \bar F$ is regularly varying at infinity with negative index, by
\cite[ Theorem 1.5.2]{bingham:goldie:teugels:1989}, the convergence of $g(\bar
F_Z(u_ns))/g(k/n)$ to $(\esp[\sigma^\alpha(X)])^{-\beta} \Tc(s)$ is uniform on
$[1,\infty)$.  Thus we have proved that $v_n(A)\to0$.
\end{proof}

\begin{proof}[Proof of Theorem \ref{theo:estimation-general}]
  Define
  \begin{gather*}
    K(B,s) = T_{\mathcal C}(s)\muc(A) \rho(A,B,m) \; , \ \
    \tilde K_n(B,s) = \frac1{ng(k/n)} \sum_{j=1}^n W_{j,n}(s) V_j(B) \; , \\
    \tilde e_n(s) = \tilde K_n(\mathbb{R}^{h'+1},s) = \frac1{ng(k/n)}    \sum_{j=1}^n W_{j,n}(s) \; , \ \
    \xi_n = \frac{Y_{(n:n-k)}}{u_n} \; .
  \end{gather*}
With this notation, we have
\begin{align*}
  \hat\rho_n(A,B,m) = \frac{\tilde K_n(B,\xi_n)}{\tilde e_n(\xi_n)}
\end{align*}
Equations (\ref{eq:asymp_of_V}) and (\ref{eq:asymp_of_W}) imply, respectively,
that $$ \lim_{n\to\infty} \esp[\tilde K_n(B,s)] = K(B,s)\; \qquad
\lim_{n\to\infty} \esp[\tilde e_n(s)] = T(s) \mu_{\mathcal C}(A).$$  With
this in mind, we split
\begin{multline}
  \hat \rho_n(A,B,m) - \rho(A,B,m) \\
  = \frac{\tilde K_n(B,\xi_n)-K(B,\xi_n)}{\tilde e_n(\xi_n)} \ - \
  \frac{\rho(A,B,m)}{\tilde e_n(\xi_n)}\{\tilde e_n(\xi_n) - \muc(A) \,
  \Tc(\xi_n) \} \; .  \label{eq:decomp-hatrho}
\end{multline}
Thus, we only need to find the correct norming sequence $w_n$ and asymptotic
distribution in $\mathcal D([a,b])$ for any $0<a<b$ of the sequence of processes
$w_n\{\tilde K_n(B,\cdot) - K(B,\cdot)\}$. To do this, define further
\begin{align}\label{eq:expected-value}
  K_n(B,s) = \esp[\tilde K_n(B,s)] \; .
\end{align}
Then
\begin{align*}
  \tilde K_n(B,s) - K(B,s) = \tilde K_n(B,s) - K_n(B,s) + K_n(B,s) - K(B,s) \; .
\end{align*}
The term $K_n(B,s) - K(B,s)$ is a deterministic bias term that will be
dealt with by the second order condition~(\ref{eq:second-ordre-adhoc}). Write $
\tilde K_n - K_n = (ng(k/n))^{-1/2} E_{n,1}+E_{n,2}$ with
\begin{align}
  E_{n,1} (B,s) & = \frac1{\sqrt{ng(k/n)}} \sum_{j=1}^n \{ W_{j,n}(s)V_j(B) -
  \esp[W_{j,n}(s) V_j(B) \mid \mathcal X] \} \; , \label{eq:iid-term}  \\
  E_{n,2}(B,s) & = \frac{1}{ng(k/n)} \sum_{j=1}^n \esp[W_{j,n}(s) V_j(B) \mid
  \mathcal X] - K_n(B,s) \nonumber  \\
  & = \frac1n \sum_{j=1}^n \{G_n(s,\mathbf X_{j,j+h-1},\mathbf X_{j+m,j+m+h'},B) -
  K_n(B,s)\} \; . \label{eq:lrd-term}
\end{align}

The term in~(\ref{eq:iid-term}) will be called the i.i.d. term. It is
a sum of conditionally independent random variables. The term
in~(\ref{eq:lrd-term}) will be called the dependent term. It is a
function of the dependent vectors $(\mathbf X_{j,j+h-1},\mathbf
X_{j+m,j+m+h'})$.

We now state some claims whose proofs are postponed to the end of this
section. The implication of Claims \ref{claim:iid-general} and
\ref{claim:weak-dependence-general} is, in particular, that in the weakly
dependent case only the i.i.d. part contributes to the limit.
\begin{claim}
  \label{claim:iid-general}
  The process $E_{n,1}$ converges in the sense of finite-dimensional
  distributions to a Gaussian process $W$ with covariance
  \begin{multline}
    (\esp[\sigma^\alpha(X_1)])^\beta    \cov(W(B,s),W(B',s')) \\
    = \esp \Big [ \covmdep_1 (A,\bsigma(\xh),\bsigma(\mathbf X_{1,h}),
    s,s') \times \pr(\mathbf Y_{m,m+h'} \in B,\mathbf Y_{m,m+h'} \in
    B' \mid \mathcal X) \Big]    \\
    + \sum_{j=2}^{h\wedge(m-h)} \esp \Big [ \covmdep_j
    (A,\bsigma(\xh),\bsigma(\mathbf X_{j,j+h-1}), s,s') \\
    \hspace*{2cm} \times \{\pr(\mathbf Y_{m,m+h'} \in B,\mathbf
    Y_{m+j-1,m+h'+j-1} \in B' \mid \mathcal X)    \\
    + \pr(\mathbf Y_{m,m+h'} \in B',\mathbf Y_{m+j-1,m+h'+j-1} \in B
    \mid \mathcal X) \} \Big] \; ,    \label{eq:covBB'}
  \end{multline}
  where the functions $\mathcal L_j$ are defined in
  Assumption~\ref{hypo:covmdep}.
\end{claim}

\begin{claim}
  \label{claim:tightness}
  For each fixed $B$, $E_{n,1}(B,\cdot)$ is tight in
  $\mathcal D([a,b])$ for each $0<a<b$.
\end{claim}
This claim is proved in Lemma~\ref{lem:tightness-iid-part}.

The previous two statements are valid in both weakly dependent and long memory
case.  The next one may not be valid in the long memory case.  See Section
\ref{sec:long-memory}.
\begin{claim}
  \label{claim:weak-dependence-general}
  In the weakly dependent case $E_{n,2}(B,\cdot)=O_P(\sqrt n)$, uniformly with
  respect to $s \in [a,b]$ for any $0<a<b$.
\end{claim}
The next claim is proved in
\cite[Corollary~2.4]{kulik:soulier:2011}.
\begin{claim}
  \label{claim:xi_n}
  $\xi_n-1=o_P(1)$.
\end{claim}
The last thing we need is the negligibility of the bias term.
\begin{claim}
    \label{claim:bias}
    For any $a>0$, $\sup_{s\geq a}\sup_{B}|K_n(B,s) - K(B,s)| =O( v_n(A))$.
\end{claim}

Therefore if $ng(k/n)\to\infty$ and~(\ref{eq:second-ordre-adhoc})
holds (i.e. $ng(k/n)v_n(A)\to 0)$, then
\begin{align*}
  \sqrt{ng(k/n)} \{ \tilde K_n(B,\cdot) - K(B,\cdot), \tilde e_n(\cdot) -
  K(\mathbb{R}^d,\cdot)\} \Rightarrow (W (B,\cdot), W(\mathbb{R}^{h'+1},\cdot)) \; .
  \end{align*}
%   By~(\ref{eq:second-ordre-adhoc}), the same convergence hold with $K_n$
%   replaced by $K$,
  This convergence and the decomposition~(\ref{eq:decomp-hatrho}) imply
  \begin{align*}
    \sqrt{ng(k/n) \muc(A)} \{ \hat\rho_n(A,B,m) - \rho(A,B,m)\}
    \to_d W(B,1) - \rho(A,B,m) W(\mathbb{R}^{h'+1},1) \; .
\end{align*}
This distribution is Gaussian.  Applying~(\ref{eq:covBB'}) and the fact that
$\rho(A,\mathbb{R}^{h'+1},m) = 1$, it is easily checked that its variance is given
by~(\ref{eq:cov-correct}).  This concludes the proof of
Theorem~\ref{theo:estimation-general}.
\end{proof}

We now prove the claims.

\begin{proof}[Proof of Claim~\ref{claim:iid-general}]
For
$j=1,\dots,n$, denote
  \begin{align*}
    \zeta_{n,j}(B,s) =
    %\frac1{\sqrt{ng(k/n)}} \mathbf{1}_{\{\mathbf Y_{j,j+h-1}\in
     % u_n(1+s)A\}} \mathbf{1}_{\{\mathbf Y_{j+m,j+m+h'} \in B\}}=
    \frac1{\sqrt{ng(k/n)}} W_{j,n}(s) V_j(B)\; .
  \end{align*}
  In order to prove our claim, we apply the central limit theorem for
  $m$-dependent random variables, see \cite{orey:1958}. Let $C(B,B',s,s')$
  denote the quantity in the right hand side of~(\ref{eq:covBB'}).  We need to
  check that
    \begin{gather}
      \cov \left( \sum_{j=1}^n \zeta_{n,j}(B,s), \sum_{j=1}^n
        \zeta_{n,j}(B',s') \mid \mathcal X \right) \to_P C(B,B',s,s') \; ,
      \label{eq:lindeberg1}      \\
      \sum_{j=1}^n \esp[\zeta_{n,j}^4(B,s) \mid \mathcal X] \to_P 0 \; .
    \label{eq:lindeberg2}
  \end{gather}
  By standard Lindeberg-Feller type arguments, this proves the one-dimensional
  convergence. The finite-dimensional convergence is proved by similar
  arguments and by computing the asymptotic covariances. We now
  prove~(\ref{eq:lindeberg1}) and~(\ref{eq:lindeberg2}).

For $u\geq1$, $\mathbf x, \mathbf x' \in \mathbb{R}^h$, denote
\begin{align*}
  \covmdep_{n,u}(A,\mathbf x, \mathbf x',s,s') & = \frac{\pr(\bsigma(\mathbf x)
    \cdot \mathbf Z_{1,h} \in u_nsA, \bsigma(\mathbf x') \cdot \mathbf
    Z_{u,u+h-1} \in u_ns'A)}{g(\bar F_Z(u_n))} \; .
\end{align*}
For $1 \leq u \leq h$, by Assumptions~\ref{hypo:cone} and~\ref{hypo:covmdep},
the functions $\covmdep_{n,u}$ converge in $L^1(\xh,\mathbf X_{u,u+h-1})$ to
the functions $\covmdep_u$ defined in Assumption~\ref{hypo:covmdep}.  For
$u>h$, $\zh$ and $\mathbf Z_{u,u+h-1}$ are independent, so $\covmdep_{n,u}$
converges a.s. and in $L^1(\xh,\mathbf X_{u,u+h-1})$ to 0.

The random variables $\zeta_{n,j}$ are $m+h'$ dependent. Thus,
\begin{eqnarray}
  \lefteqn{\cov\left(\sum_{j=1}^n \zeta_{n,j}(B,s), \sum_{j=1}^n \zeta_{n,j}(B',s') \mid \mathcal X
  \right)  = \sum_{j=1}^n \cov(\zeta_{n,j}(B,s), \zeta_{n,j}(B',s') \mid \mathcal
  X)}\nonumber
  \\
  &&\hspace*{3cm} + \sum_{j=1}^n \sum_{u=1}^{m+h'} \cov(\zeta_{n,j}(B,s),
  \zeta_{n,j+u}(B',s') \mid \mathcal X)\label{eq:evaluation-of-covariance} \\
  &&\hspace*{3cm} + \sum_{j=1}^n \sum_{u=1}^{m+h'} \cov(\zeta_{n,j+u}(B,s),
  \zeta_{n,j}(B',s') \mid \mathcal X) \; .\label{eq:evaluation-of-covariance-1}
\end{eqnarray}
For $u=1,\dots,h\wedge(m-h)$ it is easily seen that
\begin{align*}
  \sum_{j=1}^n & \cov(\zeta_{n,j}(B,s), \zeta_{n,j+u}(B',s')\mid \mathcal X) \\
  & \sim \frac{g(\bar F_Z(u_n))}{ng(k/n)} \sum_{j=1}^n \covmdep_{n,u} (\mathbf
  X_{j,j+h-1},\mathbf X_{j+u,j+u+h-1},s,s') \\
& \hspace*{3.5cm} \times \pr(\mathbf Y_{j+m,j+m+h} \in
  B , \mathbf Y_{j+u+m,j+u+m+h'} \in  B' \mid \mathcal X) \\
  & \to _P \frac{\esp \left[ \covmdep_u(A,\xh,\mathbf X_{u,h+u-1},s,s')
      \pr(\mathbf Y_{m,m+h} \in B , \mathbf Y_{u+m,u+m+h'} \in B' \mid
      \mathcal X) \right]}{(\esp[\sigma^\alpha(X)])^\beta}\; .
\end{align*}
This yields the right-hand side of~(\ref{eq:covBB'}), so we must prove that the
terms in (\ref{eq:evaluation-of-covariance}) and
(\ref{eq:evaluation-of-covariance-1}) are negligible. If $h>m-h$, then for large
$n$ and $m-h<u \leq h$, we have $(u_ns'A) \cap B=0$, so, for~all~$j=1\dots,n$,
\begin{multline*}
  \pr \big( \mathbf Y_{j,j+h-1} \in u_nsA, \mathbf Y_{j+u,j+u+h-1} \in u_ns'A,  \\
  \mathbf Y_{j+m,j+m+h} \in B , \mathbf Y_{j+u+m,j+u+m+h'} \in B' \mid \mathcal
  X \big) = 0 \; .
\end{multline*}
For $u>h$, then as mentioned above, $\covmdep_u(A,\cdot, \cdot,s, s')$
converges to 0 in $L^1(\xh,\mathbf X_{u,u+h-1})$ so
\begin{align*}
  \sum_{j=1}^n \cov(\zeta_{n,j}(B,s), \zeta_{n,j+u}(B',s')\mid \mathcal X) \to_P 0 \; .
\end{align*}
This proves~(\ref{eq:lindeberg1}). Next, since $\zeta_{n,j}$ are indicators and applying (\ref{eq:asymp_of_V})
\begin{align*}
  \sum_{j=1}^n \esp[\zeta_{n,j}^4(B,s)] \leq C \frac{ \esp[W_{1,n}(s,A)V_1(B)]}{ng(k/n)} \to 0 \; .
\end{align*}
This proves~(\ref{eq:lindeberg2}) and the weak convergence of finite
dimensional distributions.
\end{proof}

% \begin{proof}[Proof of Claim \ref{claim:tightness}]
% \tcr{plus tres utile}
% Applying Lemma~\ref{lem:borne-bill-mdep}, we obtain the following bound (where
% we write $E_{n,1}(s)$ for $E_{n,1}(B,s)$)
% \begin{align*}
%   \esp\left[(E_{n,1}(s)-E_{n,1}(t_1))^2 (E_{n,1}(t_2)-E_{n,1}(s))^2\mid \mathcal
%     X\right] \leq C \left(Q_n(t_1)-Q_n(t_2)\right)^2\; ,
% \end{align*}
% where
% $$
% Q_n(s)=\frac{1}{ng(k/n)}\sum_{j=1}^n\esp[W_{j,n}(s,A) \mid {\cal X}] \to _P T_{\mathcal C}(s)\muc(A)\; .
% $$
% Thus we can apply Lemma~\ref{lem:bill-15.6-conditionnel} and conclude that the
% sequence of processes $E_{n,1}$ is tight.
% %
% \end{proof}

\begin{proof}[Proof of Claim~\ref{claim:weak-dependence-general}]
    By definition of the functions $L_n$ and $G_n$ (cf. (\ref{eq:function-L}) and
  (\ref{eq:def-gn-general})), it clearly holds that
\begin{align*}
  |G_n(s,\mathbf X_{j,j+h-1},\mathbf X_{j+m,j+m+h'},B)|\le L_n(s,\mathbf
  X_{j,j+h-1}) \; .
\end{align*}
We apply the variance inequality (\ref{eq:variance-inequality-arcones}) in the
weak dependence case to get
\begin{align*}
  \var (E_{n,2}(B,s))\le\frac{C}{n} \var (G_n(s,\mathbf X_{1,h},\mathbf
  X_{1+m,1+m+h'},B)) \le \frac{1}{n}\esp[L_n^2(s,\mathbf X_{1,h})] \; .
\end{align*}
By~(\ref{eq:bound}), $L_n(s,\mathbf x) \leq
M_A(\boldsymbol\sigma(\mathbf x))$.  Thus, by
Assumption~\ref{hypo:moments-general}, the right hand side is
uniformly bounded, thus $\var(E_{n,2}(B,s))=O(1/n)$ and for any fixed
$s>0$, $\sqrt{n}E_{n,2}(B,s)=O_P(1)$.  Tightness follows from
Lemma~\ref{lem:tightness-lrd-part}, thus $E_{n,2}(B,\cdot)$ converges
uniformly to 0 on any compact set of $(0,\infty]$.
\end{proof}

\begin{proof}[Proof of Claim \ref{claim:bias}]
  Consider now the bias term $K_n - K$.  Recall that (see
  (\ref{eq:expected-value}) and (\ref{eq:asymp_of_V}))
$$
K_n(B,s)=\esp[\bar K_n(B,s)]\to T_{\cal C}(s)\muc (A)\rho(A,B,m)=K(B,s)
$$
Therefore, $K_n(B,s)$ converges pointwise to $K(B,s)$. The goal here is to
show that this convergence is uniform.  Using the definition of $K_n$,
(\ref{eq:function-L}) and (\ref{eq:def-gn-general}) we have
\begin{align*}
  K_n(B,s) = \esp[G_n(s,\mathbf X_{1,h},\mathbf X_{m,m+h'},B)] =
  \esp[L_n(s,\mathbf X_{1,h}) \pr(\sigma(\mathbf X_{m,m+h'}) \cdot \mathbf
  Z_{m,m+h'} \in B \mid \mathcal X)] \; .
\end{align*}
Using this definition and recalling
the formula for $\rho(A,B,m)$ (see (\ref{eq:def-rhoA_B_m}))
\begin{align*}
  K(B,s) = \Tc(s) \esp[L(\xh) \pr(\sigma(\mathbf X_{m,m+h'}) \cdot \mathbf
  Z_{m,m+h'} \in B \mid \mathcal X)] \; .
\end{align*}
Therefore, recalling the definition (\ref{eq:vitesse}) of $v_n(A)$, we obtain
that
\begin{align*}
  | K_n(B,s) - K(B,s)| & \leq \esp \left[ \sup_{s\geq1}| L_n(s,\xh) - \Tc(s)
    L(\xh)| \right] = v_n(A) \; .
\end{align*}
\end{proof}

\begin{proof}[Proof of Corollary~\ref{coro:empirical}]
  In the following, $\mathbf y$ stands for the set $(-\boldsymbol\infty,\mathbf y]$ in
  the previous notation. For $\mathbf y \in \mathbb{R}^{h'+1}$, rewrite the
  decomposition~(\ref{eq:decomp-hatrho}) in the present context to get
  \begin{align*}
    \hat\Psi_n(\mathbf y) - \Psi(\mathbf y) = \frac{\tilde K_n(\mathbf
      y,\xi_n)-K(\mathbf y,\xi_n)}{\tilde e_n(\xi_n)} - \frac{\Psi(\mathbf
      y)}{\tilde e_n(\xi_n)}\{\tilde e_n(\xi_n) - \muc(A) \, \Tc(\xi_n) \} \; .
  \end{align*}
  Thus we need only prove that the sequence of suitably normalized
  processes $\tilde K_n(s,\mathbf y) - K_n(\mathbf y, s)$ converge
  weakly to the claimed limit. The convergence of finite dimensional
  distributions follows from Theorem~\ref{theo:estimation-general} and
  the tightness follows from Lemmas~\ref{lem:tightness-iid-part}
  and~\ref{lem:tightness-lrd-part}.
\end{proof}

\begin{proof}  [Proof of Theorem \ref{theo:estimation-general-Thlrd}]
  Claims~\ref{claim:iid-general}, \ref{claim:tightness},
  \ref{claim:xi_n} and~\ref{claim:bias} hold under the assumptions of
  Theorem~\ref{theo:estimation-general-Thlrd}. Thus, the result will
  follow if we prove a modified version of Claim
  \ref{claim:weak-dependence-general}.
\begin{claim}
  \label{claim:lrd-general}
  If $2\tau(A,B)(1-H)<1$, then $\gamma_n^{-\tau(A,B)/2}
  E_{n,2}(A,B,\cdot)$ converges weakly uniformly on compact sets of
  $(0,\infty]$ to a process $\Tc \cdot Z(A,B)$ where the random
  variable $Z(A,B)$ is in a Gaussian chaos of order $\tau(A,B)$ and
  its distribution depends only on the Gaussian process~$\{X_n\}$.
\end{claim}
For any $d \in\mathbb{N}^*$, $\mathbf q \in \mathbb{N}^d$ and $\mathbf x \in \mathbb{R}^d$, denote
  \begin{align*}
    \mathbf H_{\mathbf q}(\mathbf x) = \prod_{i=1}^d H_{q_i}(x_i) \;.
  \end{align*}
 Define $\mathbb X_j =
  (X_{j+1},\dots,X_{j+h},X_{j+m},\dots,X_{j+m+h'})$.  The Hermite
  coefficients of $G_n(s,\cdot)$ and $G$ with respect to $\mathbb X_0$
  can be expressed, for $\mathbf q\in\mathbb{N}^{h+h'+1}$, as
\begin{align*}
  J_n(\mathbf q, s) = \esp[\mathbf H_{\mathbf q}(\mathbb X_0)
  G_n(s,\mathbb X_0)] \; ,\quad J(\mathbf q) = \esp[\mathbf H_{\mathbf
    q}(\mathbb X_0) G(\mathbb X_0)] \; .
\end{align*}
Since $G_n(s,\cdot)$ converges to $T(s)
G(\cdot)$ in $L^{p}(\mathbb X_0)$ for some $p>1$, $J_n(\mathbf q, s)$
converges to $\Tc(s) J(\mathbf q)$.
Let $U$ be an $(h+h'+1)\times(h+h'+1)$ matrix such that $UU'$
is equal to the inverse of the covariance matrix of $\mathbb X_0$.
Define $J_n^*(\mathbf q,s) = \esp[\mathbf H_{\mathbf q} (U\mathbb
X_0) G_n(s,U\mathbb X_0)]$ and $J^*(q) = \esp[\mathbf H_{\mathbf q}
(U\mathbb X_0) G(\mathbb X_0)]$.  Under
Assumption~\ref{hypo:hermite-rank}, the function $G_n$ can be expanded for $\mathbf x\in
\mathbb{R}^{h+h'+1}$ as
\begin{align*}
  G_n(s,\mathbf x) -\esp[G_n(s,\mathbb X_0)] = \sum_{|\mathbf q|=\tau(A,B)}
  \frac{J_n^*(\mathbf q,s)}{\mathbf q!} \mathbf H_{\mathbf q}(U \mathbf x) + r_n(s,\mathbf
  x) \; ,
\end{align*}
where $r_n$ is implicitly defined and has Hermite rank at least
$\tau(A,B)+1$ with respect to $U\mathbb X_0$. Denote $R_n(s) = n^{-1}
\sum_{j=1}^n r_n(s,\mathbb X_j)$.
Applying~(\ref{eq:variance-inequality-arcones}), we have
\begin{align*}
  \var \left( R_n(s) \right) & \leq C \left(\gamma_n^{\tau(A,B)+1}\vee
    \frac1n\right) \var(G_n(s,\mathbb X_0)) \leq C \left(\gamma_n^{\tau(A,B)+1}\vee
    \frac1n\right) \esp[L_n^2(s,\mathbf X_{1,h})] \; .
\end{align*}
By Assumption~\ref{hypo:moments-general}, $\esp[L_n^2(s,\mathbf
X_{1,h})]$ is uniformly bounded, thus
$\var(R_n(s))=o(\gamma_n^{\tau(A,B)})$ and
$\gamma_n^{-\tau(A,B)}R_n(s)$ converges weakly to zero.  The
convergence is uniform by an application of
Lemma~\ref{lem:wichura-modified}.

Thus, the asymptotic behaviour of $\gamma_n^{-\tau(A,B)/2} E_{n,2}$ is
the same as that of
\begin{align*}
  Z_n(s) = \sum_{|\mathbf q| = \tau(A,B)} \frac{J_n^*(\mathbf q,s) n^{-1}}{\mathbf
    q!}  \; \gamma_n^{-\tau(A,B)/2} \sum_{j=1}^n \mathbf H_{\mathbf q}(U \mathbb X_j) \; .
\end{align*}
By \cite[Theorem~6]{arcones:1994},
there exist random variables $\limitlrd^*(\mathbf q)$ such that
$Z_n(s)$ converges to
\begin{align*}
  \Tc(s) \sum_{|\mathbf q| = \tau(A,B)} \frac{J^*(\mathbf q) }{\mathbf q!}
  \; \limitlrd^*(\mathbf q) \;
\end{align*}
for each $s\geq0$. To prove that the convergence is uniform, we only
need to prove that $J_n^*(\mathbf q,\cdot)$ converges uniformly to
$\Tc \cdot J^*(\mathbf q)$ for each $\mathbf q$ such that $|\mathbf
q|=\tau(A)$. Since the coefficients $J_n^*$ can be expressed
linearly in terms of the coefficients $J_n$, it suffices to prove
uniform convergence of the coefficients $J_n$. Applying H\"older
inequality, we obtain, for $p>1$ and for any $a>0$,
\begin{align*}
  \sup_{s\geq a} |J_n(\mathbf q,s) - \Tc(s)J(\mathbf q)| \leq C \esp \left[
    \sup_{s\geq a} \left| L_n(s,\mathbf X_{1,h}) - T_{\mathcal C}(s)L(\mathbf X_{1,h}) \right|^p
  \right] \; .
\end{align*}
We have already seen that this last quantity converges to 0 for $p=2$ by
Assumption~\ref{hypo:moments-general}.
\end{proof}

\appendix

\begin{center}
\Large{ \bf Appendix}
\end{center}

\numberwithin{equation}{section}
\numberwithin{theorem}{section}

\section{Second order regular variation of convolutions}
Denote $A \asymp B$ if there exists positive constant $c_1$ and $c_2$ such that
$c_1 A \leq B \leq c_2 B$.
\begin{lemma}
  \label{lem:second-ordre-convol-weight}
  Let $Z_1$ and $Z_2$ be i.i.d.  non negative random variables with common
  distribution function $F$ that satisfies Assumption~\ref{hypo:2RV-Z}. Then
  \begin{align*}
    \Big| \pr(u_1Z_1+u_2Z_2>t) & - \bar F(t/u_1) - \bar F(t/u_2) \big| \leq C
    u_1^{\alpha+\epsilon} u_2^{\alpha+\epsilon} \, t^{-1} \bar F(t) \int_0^{t}
    \bar F(s) \, \mathrm{d}s \; .
  \end{align*}

\end{lemma}

\begin{proof}
Obviously, we have
\begin{eqnarray*}
  \lefteqn{\pr(u_1Z_1+u_2Z_2>t)  = \bar F(t/u_1) + \bar F(t/u_2) - \bar F(t/u_1) \bar
  F(t/u_2)} \\
  && + \pr(t/2 < u_1Z_1  \leq  t) \pr(t/2 < u_2Z_2  \leq  t) \\
  && + \pr(u_1Z_1 \leq t/2, u_2Z_2 \leq t, u_1Z_1+u_2Z_2>t)\\
  && + \pr(u_2Z_2 \leq t/2, u_1Z_1 \leq t,
  u_1Z_1+u_2Z_2>t) \; .
\end{eqnarray*}
% Thus,
% \begin{multline*}
%   \pr(Z_1 \leq t/2, Z_2 \leq t, Z_1+Z_2>t)  + \pr(Z_1 \leq t, Z_2 \leq t/2, Z_1+Z_2>t) \\
%   \leq \pr(Z_1+Z_2>t) -  \bar F_1(t) - \bar F_2(t)  \\
%   \leq \pr(Z_1 \leq t/2, Z_2 \leq t, Z_1+Z_2>t) + \pr(Z_1 \leq t, Z_2 \leq t/2,
%   Z_1+Z_2>t) + \bar F_1(t) \bar F_2(t) \; .
%  \end{multline*}
 Consider for instance the second last term. It may be written as
%\begin{align*}
%  \frac{ \pr(u_1Z_1 \leq t/2, u_2Z_2 \leq t, u_1Z_1+u_2Z_2>t)}{\bar F(t/u_2)} = \esp \left[
%    \mathbf{1}_{\{u_1Z_1 \leq t/2\}} \left\{ \frac{ \bar F(t(1-u_1Z_1/t)/u_2)}{\bar F(t/u_2)} - 1
%    \right\} \right] \; .
%\end{align*}
\begin{align*}
I_1:=\esp \left[
    \mathbf{1}_{\{u_1Z_1 \leq t/2\}} \left\{ \frac{ \bar F(t(1-u_1Z_1/t)/u_2)}{\bar F(t/u_2)} - 1
    \right\} \right] \; .
\end{align*}
Since $F$ satisfies Assumption~\ref{hypo:2RV-Z}, we have, for $u \in [1/2,1]$,
  \begin{align*}
    0 \leq \frac{ \bar F(ut)}{\bar F(t)} - 1 & = u^{-\alpha} \mathrm{e}^{ \int_1^u
      \frac{\eta(ts)}s \, \mathrm{d}s} - 1 = \{ u^{-\alpha} - 1 \} \mathrm{e}^{ \int_1^u
      \frac{\eta(ts)}s \, \mathrm{d}s} + \mathrm{e}^{ \int_1^u \frac{\eta(ts)}s \, \mathrm{d}s} - 1 \\
    & \leq |u^{-\alpha}-1| \mathrm{e}^{\int_{1/2}^1 \frac{\eta^*(ts)}s \, \mathrm{d}s} +
    \mathrm e^{ \int_{1/2}^1 \frac{\eta^*(ts)}s \, \mathrm{d}s} \int_{u}^1
    \frac{\eta^*(ts)}s \, \mathrm{d}s \; .
  \end{align*}
  Since $\eta^*(t)$ is decreasing, we have, for all  $u\in[1/2,1]$,
\begin{align*}
  0 \leq \frac{ \bar F(ut)}{\bar F(t)} - 1 & \leq C \{ |u^{-\alpha}-1| +
  \log(u) \} \leq C (1-u) \; .
\end{align*}
Applying this inequality with $1-u=u_1Z_1/t$ on the event $u_1Z_1\leq t/2$ yields
\begin{align*}
  I_1 \leq C u_1t^{-1}
  \esp \left[ Z_1 \mathbf{1}_{\{u_1Z_1 \leq t\}} \right] \leq C t^{-1} \int_0^{t/u_1} \bar
  F(s) \, \mathrm{d}s =Ct^{-1}u_1^{-1}\int_0^t\bar F(s/u_1)\, \mathrm{d}s\; .
\end{align*}
By Potter's bounds, for any $\epsilon>0$, there exists a constant $C$ such for
any $s,t>0$,
\begin{align*}
  \frac{\bar F(s/u_1)}{\bar F(s)} \leq C (u_1^{-1}\wedge 1)^{-\alpha-\epsilon} \; .
\end{align*}
Applying this  bound we obtain
\begin{align*}
  I_1 \leq C
  (u_1\vee1)^{\alpha+\epsilon} (u_2\vee1)^{\alpha+\epsilon} t^{-1}\bar F(t)
  \int_0^{t} \bar F(s) \, \mathrm{d}s \; .
\end{align*}
To conclude, note that $\bar F^2(t) = O(t^{-1}\bar F(t) \int_0^{t} \bar F(s)
\, \mathrm{d}s)$ if $\alpha<1$ and $\bar F^2(t)$ = $o(t^{-1}\bar F(t) \int_0^{t} \bar
F(s) \, \mathrm{d}s)$ if $\alpha\geq1$.
\end{proof}

\begin{remark}
  By induction, we can obtain the bound
  \begin{align*}
    \Big| \pr(Z_1+ \cdots + Z_n>t) - n \bar F(t) \big| \leq C \, t^{-1} \bar
    F(t) \int_0^{t} \bar F(s) \, \mathrm{d}s \; ,
  \end{align*}
  and we can also recover  a particular case of a result of
  \cite{omey:willekens:1987} in a slightly different
  form. For $\alpha\geq1$ and $\esp[Z_1]<\infty$,
  \begin{align*}
    \lim_{t \to \infty} t \Big\{ \frac{\pr(Z_1+ \cdots + Z_n > t)} {\pr(Z_1>t)}
    - n \Big \} = \frac{n(n-1)}2 \esp[Z_1] \; .
  \end{align*}
\end{remark}

\section{Multivariate Hermite expansions and variance inequalities for Gaussian processes}
\label{sec:hermite}
Consider a multidimensional stationary centered Gaussian process $\{\mathbf
X_n\}$ with autocovariance function $\gamma_n(i,j) = \esp[ X_0^{(i)} X_n^{(j)}]$
and assume either
\begin{align}
\label{eq:weak-dependence}
\forall 1 \leq i,j \leq d \;, \ \ \sum_{n=0}^\infty |\gamma_n(i,j)| < \infty \; ,
\end{align}
or that there exists $H\in(1/2,1)$ and a function $\ell$ slowly varying at
infinity such that
\begin{align}
  \label{eq:lrd-1}
  \lim_{n\to\infty} \frac{\gamma_n(i,j)}{n^{2H-2}\ell(n)} = b_{i,j} \; ,
\end{align}
and the coefficients  $b_{i,j}$ are not identically zero. Then, we have the following
inequality due to \cite{arcones:1994}.

%  If~(\ref{eq:lrd-1}) holds and $2q(1-H)<1$, then
For any function $G$ such that $\esp[G^2(\mathbb X_0)]<\infty$ and with Hermite
rank $q$ with respect to $\mathbf X_0$,
\begin{align}
  \label{eq:variance-inequality-arcones}
  \var\left( n^{-1} \sum_{j=1}^n G(\mathbf X_j) \right) \leq C (\ell^q(n)n^{2q(H-1)})
  \vee n^{-1} \, \var(G(\mathbf X_0)) \; .
\end{align}
% \item If~(\ref{eq:lrd-1}) holds and $2q(1-H)>1$, then for any function $G$ with
%   Hermite rank $q$ with respect to $\mathbf X_0$,
% \begin{align}
%   \label{eq:variance-inequality-weak-dependence}
%   \var\left( \sum_{j=1}^n G(\mathbf X_j) \right) \leq C n \; \var(G(\mathbf X_0))
%   \; .
% \end{align}
% \item If~(\ref{eq:weak-dependence}) holds,
%   then~(\ref{eq:variance-inequality-weak-dependence}) still holds.
% \end{itemize}
where the constant $C$ depends only on the Gaussian process $\{\mathbf X_n\}$
and not on the function $G$. This bound summarizes Equations~2.18, 3.10 and~2.40
in \cite{arcones:1994}. The rate obtained is $n^{-1}$ in the weakly dependent
case where (\ref{eq:weak-dependence}) holds and in the case where
(\ref{eq:lrd-1}) holds and $G$ has Hermite rank $q$ such that $q(1-H)>1$.
Otherwise, the rate is $\ell^q(n)n^{2q(H-1)}$.

\section{A criterion for tightness}
\label{app:tightness}

We state a criterion for the tightness of a sequence of random
processes with path in $\mathcal D(\mathbb{R}^d)$, which adapts to the
present context \citet[Theorem 3]{bickel:wichura:1971} and the remarks
thereafter.

Let $T$ be a rectangle $T=T_1\times T_d \subset \mathbb{R}^d$.  A block $B$ in $T$ is
a subset of $T$ of the form $\prod_{i=1}^d (s_i,t_i]$ with $s_i<t_i$, $1 \leq i
\leq d$. Disjoint blocks $B = \prod_{i=1}^d (s_i,t_i]$ and $B'=\prod_{i=1}^d
(s'_i,t'_i]$ are neighbours if there exists $p\in\{1,\dots,d\}$ such that
$s'_p=t_p$ or $s_p=t'_p$ and $s_i=s'_i$ and $t_i=t'_i$ for $i\ne p$. (In the
terminology of \cite{bickel:wichura:1971} the blocks $B$ and $B'$ are said to
share a common face.) Let $X$ be a random process indexed by $T$. The increment
of the process $X$ over a block $B=\prod_{i=1}^d(s_i,t_i]$ is defined by
\begin{align*}
  X(B) = \sum_{(\epsilon_1,\dots,\epsilon_d)\in\{0,1\}^d} (-1)^{d-\sum_{i=1}^d
    \epsilon_i} X(s_1+\epsilon_1(t_1-s_1),\dots,s_d+\epsilon_d(t_d-s_d))  \; .
\end{align*}
(This is the usual $d$-dimensional increment of a random process $X$. If for
instance $d=2$, then $X(B) = X(t_1,t_2) - X(t_1,s_2) - X(s_1,t_2) +
X(s_1,s_2)$). If $X$ is an indicator, i.e. $X(\mathbf y) = \mathbf{1}_{\{\mathbf Y\leq
  \mathbf y\}}$ for some $T$ valued random variable $\mathbf Y$, then $X(B) =
\mathbf{1}_{\{\mathbf Y \in B\}}$.

\begin{lemma}
  \label{lem:wichura-modified}
  Let $\{\zeta_{n}\}$ be sequence of stochastic processes indexed by a
  compact rectangle $T \subset \mathbb{R}^d$. Assume that the finite
  dimensional marginal distributions of $\zeta_n$ converges weakly to
  those of a process $\zeta$ which is continuous on the upper boundary
  of $T$. Assume moreover that there exist $\gamma\geq0$ and $\beta>1$
  such that
  \begin{align}
    \label{eq:bickel-wichura}
    \pr(|\zeta_n(B)| \wedge |\zeta_n(B')| \geq \lambda) \leq C
    \lambda^{-\gamma} \esp[\mu_n^\beta(B\cup B')]
  \end{align}
  for some sequence of random probability measures $\mu_n$ which
  converges weakly in probability to a (possibly random) probability
  measure $\mu$ with (almost surely) continuous marginals. Then the
  sequence of processes $\{\zeta_n\}$ is tight in $\mathcal
  D(T,\mathbb{R})$.
\end{lemma}

\begin{proof}[Sketch of proof]
  For $f$ defined on $T = T_1\times\dots\times T_d$, $i\in\{1,\dots,d\}$
  and $t \in T_i$, define $f^{(i)}_t$ on $T_1 \times \dots \times
   T_{i-1} \times T_{i+1} \times \dots \times T_d$ by
  $$
  f_t^{(i)}(t_1,\dots,t_{i-1},t_{i+1},\dots,t_d) =
  f(t_1,\dots,t_{i-1},t,t_{i+1},\dots,t_d)
$$
and define, for $s<t \in T_i$ and $\delta>0$,
\begin{align*}
  w_i^{\prime\prime}(f,s,t) & = \sup_{s<u<v<w<t}
  \|f_u^{(i)}-f^{(i)}_v\|_\infty \wedge \|f_v^{(i)}-f^{(i)}_w\|_\infty  \; , \\
  w_i^{\prime\prime}(f,\delta) & = \sup_{u<v<w<u+\delta}
  \|f_u^{(i)}-f^{(i)}_v\|_\infty \wedge \|f_v^{(i)}-f^{(i)}_w\|_\infty  \; .
\end{align*}
By the Corollary of \cite{bickel:wichura:1971}, a sequence of
processes $\{X_n\}$ defined on $T$ converges weakly in $\mathcal D(T)$
to a process $X$ which is continuous at the upper boundary of $T$ with
probability one, if the finite-dimensional marginal distributions of
$X_n$ converges to those of $X$ and if, for all $\delta,\lambda>0$, and al
$i=1,\dots,d$,
\begin{align}
 \label{eq:B-W-10}
  \pr(w_i^{\prime\prime}(X_n,\delta)>\lambda) \to 0 \; .
\end{align}
For any measure $\mu$ on $T$, define its $i$-th marginal $\mu^{(i)}$ by
\begin{align*}
  \mu^{(i)}((s,t]) = \mu(T_1 \times \dots \times T_{i-1} \times (s,t]
  \times T_{i+1} \times \dots \times T_d) \; ,  s,t \in T_i \; .
\end{align*}
As mentioned in the remarks after the proof of
\citet[Theorem~3]{bickel:wichura:1971}, an easy  adaptation of the
proof of \citet[Theorem~15.6]{billingsley:1968} shows
that~(\ref{eq:B-W-10}) is implied~by
\begin{align}
 \label{eq:B-W-11}
  \pr(w_i^{\prime\prime}(X_n,s,t) > \lambda) \leq C \lambda^{-\gamma}
  \esp[\{\mu_n^{(i)}(s,t])\}^\beta] \; ,
\end{align}
where $\mu_n$ satisfies the assumptions of the Lemma.  So we must show
that~(\ref{eq:bickel-wichura}) implies~(\ref{eq:B-W-11}).  The proof
is by induction, so the first step is to prove it in the
one-dimensional case, where (\ref{eq:bickel-wichura}) becomes, for
$u<v<w \in T$,
  \begin{align}
   \label{eq:new-1}
    \pr(|\zeta_n(v)-\zeta_n(u)| \wedge |\zeta_n(w)-\zeta_n(v)| \geq \lambda) \leq C
    \lambda^{-\gamma} \esp[\mu_n^\beta((u,w])] \; .
  \end{align}
  The proof of~(\ref{eq:B-W-11}) under the assumption~(\ref{eq:new-1})
  follows the lines of the proof of \cite[(15.26)]{billingsley:1968}
  under the assumption \cite[(15.21)]{billingsley:1968}.  The key
  ingredient is the maximal inequality
  \cite[Theorem~12.5]{billingsley:1968}, which can be easily adapted
  as follows in the present context. Let $S_0,\dots,S_n$ be random
  variables. Assume that there exists nonnegative random variables
  $u_1,\dots,u_n$ such that
  \begin{align*}
    \pr(|S_i-S_j|\wedge|S_k-S_j|> \lambda) \leq \lambda^{-\gamma}
    \esp[(u_i+\dots+u_k)^\beta ]
  \end{align*}
  for some $\beta>1$ and $\gamma\geq0$ and all $1 \leq i \leq j \leq
  k\leq n$ and, then there exists a constant $C$ that depends only on
  $\beta$ and $\gamma$ such that
  \begin{align*}
    \pr \left( \max_{1 \leq i \leq j \leq k \leq n} |S_i-S_j|\wedge|S_k-S_j|>
      \lambda \right) \leq C \lambda^{-\gamma} \esp[(u_1+\dots+u_n)^\beta] \; .
  \end{align*}
  Proving by induction that~(\ref{eq:bickel-wichura})
  implies~(\ref{eq:B-W-11}) in the $d$-dimensional case can be done
  exactly along the lines of Step 5 of the proof of
  \citet[Theorem~1]{bickel:wichura:1971}.
\end{proof}

In order to apply this criterion to the context of empirical
processes, we need the following Lemma which slightly extends the
bound \citet[(13.18)]{billingsley:1968}.

\begin{lemma}
\label{lem:borne-bill-mdep}
  Let $\{(B_i,B'_i)\}$ be a sequence of $m$-dependent vectors, where $B_i$
  and $B'_i$ are Bernoulli random variables, with parameters $p_i$
  and $q_i$, respectively, and such that $B_iB'_i=0$ a.s. Denote $S_n = \sum_{j=1}^n (B_j-p_j)$ and $S'_n =
  \sum_{j=1}^n (B'_j-q_j)$. Then, there exists a constant $C$ which depends only
  on $m$, such that
  \begin{align}
    \esp[S_n^2{S_n'}^2] \leq C \left( \sum_{i=1}^n p_i \right)
    \left(\sum_{i=1}^n q_i \right) \leq C \left( \sum_{i=1}^n p_i\vee q_i \right)^2  \; .
    \label{eq:borne-bill-mdep}
  \end{align}
\end{lemma}

\begin{proof}
  We start by assuming that the pairs $(B_i,B'_i)$ are i.i.d. and we
  prove~(\ref{eq:borne-bill-mdep}) by induction. For any integrable random variable
  $X$, denote $\bar X = X - \esp[X]$. For $n=1$, since $B_1B'_1=0$, we obtain
  $\esp[\bar B_i \bar B'_i] = -p_iq_i$ and
  \begin{align*}
    \esp[\bar B_1^2 { \bar{B'}_1}^2] & = \esp[(B_1-2p_1B_1+p_1^2)
    (B'_1-2q_1B'_1+q_1^2)]    \\
    &= p_1q^2 + p_1^2 q_1 -3 p_1^2q^2 = p_1q_1(p_1+q_1 -3p_1q_1)  \leq p_1q_1 \; .
  \end{align*}
  The last inequality comes from the fact that $B_1B_1'=0$ a.s. implies that
  $p_i+q_i\leq 1$, and $0 \leq p+q -3pq \leq p+q \leq 1$ for all $p,q\geq0$ such
  that $p+q\leq 1$. Assume now that~(\ref{eq:borne-bill-mdep}) holds with $C=3$ for
  some $n\geq1$. Then, denoting $s_n=\sum_{j=1}^n p_j$ and $s'_n = \sum_{j=1}^n
  q_j$, we have
  \begin{align*}
    &    \esp[S_{n+1}^2 {S'_{n+1}}^2] \\
    & = \esp[S_{n}^2 {S'_{n}}^2] + \esp[S_n^2] \esp[{\bar{ B'}_{n+1}}^2] +
    \esp[{S'_n}^2] \esp[{\bar{ B}_{n+1}}^2] + 4
    \esp[S_{n} S'_{n}] \esp[B_{n+1} B'_{n+1}] +     \esp[\bar B_{n+1}^2 {\bar{B'}_{n+1}}^2] \\
    & \leq 3 s_ns'_n + s_n q_{n+1} + s'_n p_{n+1} + 4 p_{n+1}q_{n+1}
    \sum_{i=1}^n p_iq_i + p_{n+1}q_{n+1} \\
    & \leq 3 s_ns'_n + 3s_n q_{n+1} + 3s'_n p_{n+1} + p_{n+1}q_{n+1} \leq 3
    s_{n+1} s'_{n+1} \; .
  \end{align*}
  This proves that~(\ref{eq:borne-bill-mdep}) holds for al $n\geq1$.

  We now consider the case of $m$-dependence.  Let $a_i$, $1 \leq i \leq n$ be a
  sequence of real numbers and set $a_i=0$ if $i>n$. Then
\begin{align*}
  \left( \sum_{i=1}^n a_i \right)^2 & = \left(\sum_{q=1}^m \sum_{j=1}^{\lceil
      n/m \rceil} a_{(j-1)m+q} \right)^2 \leq m \sum_{q=1}^m
  \left(\sum_{j=1}^{\lceil n/m \rceil} a_{(j-1)m+q} \right)^2 \; .
\end{align*}
Applying this and the bound for the independent case (extending all
sequences by zero after the index $n$) yields
\begin{align*}
  \esp[S_{n}^2 {S'_{n}}^2] \leq 3m^2 \sum_{q=1}^m \sum_{q'=1}^m \sum_{j=1}^{\lceil
    n/m \rceil} \sum_{j'=1}^{\lceil n/m \rceil} p_{(j-1)m+q}p_{(j'-1)m+q'} = 3m^2 s_ns'_n \; .
\end{align*}
\end{proof}

Let us apply this criterion in the context of
section~\ref{sec:estimation}. Fix a cone $\mathcal C$ and a relatively
compact subset $A \in \mathcal C$. Recall that $E_{n,1}$ and $E_{n,2}$
are defined in~(\ref{eq:iid-term}) and~(\ref{eq:lrd-term}).

\begin{lemma}
  \label{lem:tightness-iid-part}
  Under the assumptions of Theorem~\ref{theo:estimation-general}
  or~\ref{theo:estimation-general-Thlrd}, for any fixed
  $B\in\mathbb{R}^{h'+1}$, $E_{n,1}(B,\cdot)$ is tight in $\mathcal
  D([a,b])$, and if moreover $\Psi_{A,m,h}$ is continuous, then
  $E_{n,1}$ is tight in $\mathcal D(\mathcal K \times [a,b])$ for any
  $0<a<b$ and any compact set $\mathcal K$ of $\mathbb{R}^{h'+1}$.
\end{lemma}
\begin{proof}
  By Assumption~\ref{hypo:simplification}, if $s<t$, then $tA \subset
  sA$. Thus, a sequence of random measures $\hat\mu_n$ on $ \mathbb{R}^d
  \times (0,\infty)$ can be defined by
\begin{align*}
  \hat \mu_n((-\boldsymbol\infty,\mathbf y] \times (s,\infty) ) & = \frac1n
  \sum_{j=1}^n \frac{\pr(\mathbf Y_{j,h} \in s u_n A \mid \mathcal X)}{g(k/n)} \pr(\mathbf
  Y_{j+m,j+m+h'} \leq \mathbf y \mid \mathcal X) \\
& = \frac1n   \sum_{j=1}^n G_n(s,\mathbf X_{j,h} \mathbf
  X_{j+m,j+m+h'}, \mathbf y ) \; ,
\end{align*}
where $G_n$ is defined in~(\ref{eq:def-gn-general}).  Then $\hat\mu_n$
converges vaguely in probability to the measure $\mu$ defined by
\begin{align*}
  \mu( (-\boldsymbol\infty,\mathbf y]\times(s,\infty)) = \muc(A) T(s)
  \Psi_{A,m,h}(\mathbf y) \; .
\end{align*}
Then, by conditional $m$-dependence, for any neighbouring relatively
compact blocs $D,D'$ of $ \mathbb{R}^d\times(0,\infty]$, applying
Lemma~\ref{lem:borne-bill-mdep} yields
\begin{align*}
  \esp[ E_{n,1}^2(D) E_{n,2}^2(D') \mid \mathcal X] \leq C \hat \mu_n(D) \hat
  \mu_n(D') \; .
\end{align*}
Taking unconditional expectations then yields
\begin{align*}
  \esp[ E_{n,2}^2(D) E_{n,2}^2(D')] \leq C \hat \esp[\mu_n(D) \hat
  \mu_n(D')] \leq \esp[\hat\mu_n^2(D \cup D')] \; .
\end{align*}
Thus~(\ref{eq:bickel-wichura}) holds with $\beta=\gamma=2$. In the
context of Theorem~\ref{theo:estimation-general}, for any fixed $B$,
this implies that $E_{n,1}(B,\cdot)$ is fixed, since the limiting
distribution is proportional to $T(s)$ which is
continuous. If the distribution function $\Psi$ is assumed to be
continuous, then Lemma~\ref{lem:wichura-modified} applies and the
process $E_{n,1}$ is tight with respect to both variables.
\end{proof}

\begin{lemma}
 \label{lem:tightness-lrd-part}
 Under the assumptions of Theorem~\ref{theo:estimation-general}, for any fixed $B
 \in\mathbb{R}^{h'+1}$, $E_{n,2}(B,\dot)$ converges uniformly to zero on
 compact sets of $(0,\infty]$. Under the assumption of
 Corollary~\ref{coro:empirical}, $E_{n,2}$ converges uniformly to zero
 on compact sets of $ \mathbb{R}^{h'+1} \times (0,\infty]$.
\end{lemma}

\begin{proof}
  We only need to prove the tightness.  By the variance
  inequality~(\ref{eq:variance-inequality-arcones}) and H\"older's
  inequality, we have, for any relatively compact neighbouring blocks
  $D,D'$ of $\mathbb{R}^d\times(0,\infty)$,
\begin{align*}
  \pr(|E_{2,n}(D)| \wedge |E_{2,n}(D')| \geq \lambda) & \leq \lambda^{-2}
  \sqrt{\esp[E_{2,n}^2(D)]\esp[E_{2,n}^2(D')]} \leq \lambda^{-2}
  \esp[E_{2,n}^2(D\cup D')] \\
  & \leq C \lambda^{-2} n^{-1} \esp[\tilde \mu_n^2(D\cup D')]
\end{align*}
where $\tilde \mu_n$ is the random measure defined by
\begin{align*}
  \tilde\mu_n(s,\mathbf y) = \frac{\pr(\mathbf Y_{1,h} \in su_nA \mid \mathcal X)}{g(k/n)}
  \pr(\mathbf Y_{m,m+h'} \leq \mathbf y \mid \mathcal X) \; .
\end{align*}
Assumptions~\ref{hypo:cone} and~\ref{hypo:moments-general} imply that $\tilde
\mu_n$ converges vaguely on $\mathbb{R}^d\times(0,\infty]$, in probability and in the
mean square to the measure $\hat\mu$ defined by
\begin{align*}
  \hat \mu ( (-\boldsymbol\infty, \mathbf y]\times(s,\infty]) = \frac
  {\nu_{\mathcal C}(\boldsymbol\sigma(\xh)^{-1}\cdot A)}
  {(\esp[\nu_C(\boldsymbol\sigma(\xh)^{-1}\cdot A])^\beta} T(s) \pr(\mathbf
  Y_{m,m+h'} \leq \mathbf y \mid \mathcal X) \; .
\end{align*}
The measure $\hat\mu$ has continuous marginals if we consider the case
of a fixed $B$ (which takes care of
Theorem~\ref{theo:estimation-general-Thlrd}). The marginals of
$\hat\mu$ are almost surely continuous if $F_Z$ is continuous, so
Lemma~\ref{lem:wichura-modified} applies.
\end{proof}

\bibliography{bibSV-1}{}

\begin{thebibliography}{16}
\providecommand{\natexlab}[1]{#1}
\providecommand{\url}[1]{\texttt{#1}}
\expandafter\ifx\csname urlstyle\endcsname\relax
  \providecommand{\doi}[1]{doi: #1}\else
  \providecommand{\doi}{doi: \begingroup \urlstyle{rm}\Url}\fi

\bibitem[Arcones(1994)]{arcones:1994}
Miguel~A. Arcones.
\newblock Limit theorems for nonlinear functionals of a stationary {G}aussian
  sequence of vectors.
\newblock \emph{The Annals of Probability}, 22\penalty0 (4):\penalty0
  2242--2274, 1994.

\bibitem[Bickel and Wichura(1971)]{bickel:wichura:1971}
Peter~J. Bickel and Michael~J. Wichura.
\newblock Convergence criteria for multiparameter stochastic processes and some
  applications.
\newblock \emph{Annals of Mathematical Statistics}, 42:\penalty0 1656--1670,
  1971.

\bibitem[Billingsley(1968)]{billingsley:1968}
Patrick Billingsley.
\newblock \emph{Convergence of probability measures}.
\newblock New York, Wiley, 1968.

\bibitem[Bingham et~al.(1989)Bingham, Goldie, and
  Teugels]{bingham:goldie:teugels:1989}
Nicholas~H. Bingham, Charles~M. Goldie, and Jan~L. Teugels.
\newblock \emph{Regular variation}, volume~27 of \emph{Encyclopedia of
  Mathematics and its Applications}.
\newblock Cambridge University Press, Cambridge, 1989.

\bibitem[Breidt et~al.(1998)Breidt, Crato, and
  de~Lima]{breidt:crato:delima:1998}
F.~Jay Breidt, Nuno Crato, and Pedro de~Lima.
\newblock The detection and estimation of long memory in stochastic volatility.
\newblock \emph{Journal of Econometrics}, 83\penalty0 (1-2):\penalty0 325--348,
  1998.

\bibitem[Breiman(1965)]{breiman:1965}
Leo Breiman.
\newblock On some limit theorems similar to the arc-sine law.
\newblock \emph{Theory of Probability and Applications}, 10:\penalty0 323--331,
  1965.

\bibitem[Davis and Mikosch(2009)]{davis:mikosch:2009}
Richard~A. Davis and Thomas Mikosch.
\newblock The extremogram: A correlogram for extreme events.
\newblock \emph{Bernoulli}, 38A:\penalty0 977--1009, 2009.
\newblock Probability, statistics and seismology.

\bibitem[Drees(1998)]{drees:1998}
Holger Drees.
\newblock Optimal rates of convergence for estimates of the extreme value
  index.
\newblock \emph{The Annals of Statistics}, 26\penalty0 (1):\penalty0 434--448,
  1998.

\bibitem[Harvey(1998)]{harvey:1998}
Andrew~C. Harvey.
\newblock Long memory in stochastic volatility.
\newblock In \emph{{\em J. Knight and S. Satchell (eds)}, Forecasting
  volatility in financial markets}. Butterworth-Heinemann, London, 1998.

\bibitem[Hurvich et~al.(2005)Hurvich, Moulines, and
  Soulier]{hurvich:moulines:soulier:2005}
Clifford~M. Hurvich, Eric Moulines, and Philippe Soulier.
\newblock Estimating long memory in volatility.
\newblock \emph{Econometrica}, 73\penalty0 (4):\penalty0 1283--1328, 2005.

\bibitem[Kulik and Soulier(2011)]{kulik:soulier:2011}
Rafa{\l} Kulik and Philippe Soulier.
\newblock The tail empirical process for long memory stochastic volatility
  sequences.
\newblock \emph{Stochastic Processes and their Applications}, 121\penalty0
  (1):\penalty0 109 -- 134, 2011.

\bibitem[Omey and Willekens(1987)]{omey:willekens:1987}
Edward Omey and Eric Willekens.
\newblock Second-order behaviour of distributions subordinate to a distribution
  with finite mean.
\newblock \emph{Communications in Statistics. Stochastic Models}, 3\penalty0
  (3):\penalty0 311--342, 1987.

\bibitem[Orey(1958)]{orey:1958}
Steven Orey.
\newblock A central limit theorem for {$m$}-dependent random variables.
\newblock \emph{Duke Mathematical Journal}, 25:\penalty0 543--546, 1958.

\bibitem[Resnick(1987)]{MR900810}
Sidney~I. Resnick.
\newblock \emph{Extreme values, regular variation, and point processes}.
\newblock Springer-Verlag, New York, 1987.

\bibitem[Resnick(2007)]{resnick:2007}
Sidney~I. Resnick.
\newblock \emph{Heavy-tail phenomena}.
\newblock Springer Series in Operations Research and Financial Engineering.
  Springer, New York, 2007.
\newblock Probabilistic and statistical modeling.

\bibitem[Resnick(2008)]{resnick:2008}
Sidney~I. Resnick.
\newblock Multivariate regular variation on cones: application to extreme
  values, hidden regular variation and conditioned limit laws.
\newblock \emph{Stochastics}, 80\penalty0 (2-3):\penalty0 269--298, 2008.

\end{thebibliography}

\bibliographystyle{plainnat}

\end{document}